\documentclass[a4paper, reqno]{amsart}
\usepackage{amssymb, amsmath, amsthm, chngcntr, enumerate, mathrsfs, mathtools, tikz}

\usetikzlibrary{arrows,calc}

\numberwithin{equation}{section}
\newtheorem{thm}{Theorem}
\newtheorem{cor}[thm]{Corollary}
\newtheorem{lem}[thm]{Lemma}
\newtheorem{prop}[thm]{Proposition}

\theoremstyle{definition}
\newtheorem{dfn}[thm]{Definition}
\theoremstyle{definition}
\newtheorem{rmk}[thm]{Remark}
\theoremstyle{definition}

\theoremstyle{definition}
\newtheorem*{notation}{Notation}

\newcommand{\C}{\mathbb{C}}
\newcommand{\R}{\mathbb{R}}

\newcommand{\N}{\mathbb{N}}
\newcommand{\ud}{\mathrm{d}}

\DeclarePairedDelimiter{\floor}{\lfloor}{\rfloor}
\DeclarePairedDelimiter{\ceil}{\lceil}{\rceil}
\DeclarePairedDelimiter{\bracket}{\langle}{\rangle}

\counterwithout{equation}{section}
\counterwithout{section}{section}
\begin{document}

\title[Dilated Averages over Polynomial Curves]{Uniform $L_x^p - L^q_{x,r}$ Improving for Dilated Averages over Polynomial Curves}

\author{Jonathan Hickman}

\address{Jonathan Hickman, Room 5409, James Clerk Maxwell Building, University of Edinburgh, Peter Guthrie Tait Road, Edinburgh, EH9 3FD.}
\email{j.e.hickman@sms.ed.ac.uk}

\begin{abstract} Numerous authors have considered the problem of determining the Lebesgue space mapping properties of the operator $\mathcal{A}$ given by convolution with affine arc-length measure on some polynomial curve in Euclidean space. Essentially, $\mathcal{A}$ takes weighted averages over translates of the curve. In this paper a variant of this problem is discussed where averages over both translates and dilates of a fixed curve are considered. The sharp range of estimates for the resulting operator is obtained in all dimensions, except for an endpoint. The techniques used are redolent of those previously applied in the study of $\mathcal{A}$. In particular, the arguments are based upon the refinement method of Christ, although a significant adaptation of this method is required to fully understand the additional smoothing afforded by averaging over dilates. 

\end{abstract}

\maketitle

\section{Introduction and statement of results}

Let $\gamma: I \rightarrow \R^d$ denote a (parametrisation of a) smooth curve in $\R^d$ where $I \subseteq \R$ is an interval. Consider the operator $A$ defined, at least initially, on the space of for all test functions $f$ on $\R^d$ by
\begin{equation}\label{operator}
Af(x,r) := \int_I f(x-r \gamma (t)) \alpha(t)\,\ud t \qquad \textrm{for all $(x,r) \in \R^d \times [1,2]$.}
\end{equation}
Here $\alpha$ denotes some density which is assumed to be smooth and non-negative. Thus $A$ takes averages over translates of dilates of $\gamma$. A natural problem is to establish the range of $(p,q_1,q_2)$ for which there is an a priori mixed norm estimate either of the form
\begin{equation}\label{mixed1}
\|Af\|_{L^{q_2}_tL^{q_1}_x(\R^d \times [1,2])} \leq C_{d, \gamma} \|f\|_{L^p(\R^d)}
\end{equation}
or
\begin{equation}\label{mixed2}
\|Af\|_{L^{q_1}_xL^{q_2}_t(\R^d \times [1,2])} \leq C_{d, \gamma} \|f\|_{L^p(\R^d)}.
\end{equation}
This question subsumes the study of the $L^p$ mapping properties of both single averages and maximal functions associated to space curves. The archetypical case to consider is when $\gamma := h \colon [0,1] \to \R^d$ is the so-called \emph{moment curve} given by 
\begin{equation}
h(t) := (t, t^2, \dots, t^d)
\end{equation}
with constant density $\alpha \equiv 1$. With this choice of curve and density the following results are known:

\begin{itemize}
\item Taking $q_2 = \infty$ in \eqref{mixed1} reduces matters to determining the set of $(p,q)$ for which the single averages
\begin{equation}\label{single}
\mathcal{A}_rf(x) := \int_0^1 f(x - rh(t))\,\ud t
\end{equation}
are type $(p,q)$ uniformly for $r \in [1,2]$. The $L^p-L^q$ mapping properties of $\mathcal{A}:= \mathcal{A}_1$ were investigated in low dimensions in a number of papers \cite{Littman1973, Oberlin1987, Oberlin1997, Greenleaf1999, Oberlin1999} before being completely determined in all dimensions by Stovall \cite{Stovall2009} using powerful new methods developed by Christ \cite{Christb, Christ1998}.

\item On the other hand, if one sets $q_2 = \infty$ in \eqref{mixed2} the situation is very different. In particular, one now wishes to understand the $L^p-L^q$ mapping properties of the maximal function $\mathcal{M}$ associated to $h$, defined by 
\begin{equation*}
\mathcal{M}f(x) := \sup_{1 \leq r \leq 2} |\mathcal{A}_rf(x)|.
\end{equation*}
A celebrated theorem of Bourgain \cite{Bourgain1986} established $L^p-L^p$ mapping properties for $d=2$; this result was extended by Schlag \cite{Schlag1997a} who proved an almost-sharp range of $L^p-L^q$ estimates.\footnote{More precisely, both Bourgain and Schlag studied the circular maximal function rather than the parabolic variant discussed here. However, in this context both objects can be understood via the same techniques.} However, the problem of determining even the $L^p-L^p$ range remains open in all other dimensions. Some partial results in this direction are given in \cite{Pramanik2007}.\footnote{It is remarked that this brief survey is far from complete: there are many other results and, in particular, an extensive literature investigating these problems for more general classes of curves.}
\end{itemize}

In this paper the special case of \eqref{mixed1} and \eqref{mixed2} where $q_1 = q_2$ is considered. In particular, Theorem \ref{momentthm} below almost completely determines the set of $(p,q)$ for which $A$ is bounded from $L^p_x(\R^d)$ to $L^q_{x,r}(\R^d \times [1,2])$ when $\gamma(t) := h(t)$ is the moment curve. Testing the inequality on some simple examples (see Section \ref{necessity}) shows such a bound is possible only if $(1/p, 1/q)$ lies in the trapezium $\mathcal{T}_d$ given by the closed, convex hull of the set 
\begin{equation*}
\{(0, 0), (1,1), (1/p_1, 1/q_1), (1/p_2, 1/q_2)\}
\end{equation*}
where
\begin{equation*}
\left(\frac{1}{p_1}, \frac{1}{q_1} \right):= \left(\frac{1}{d}, \frac{d-1}{d(d+1)}\right) \quad \textrm{and} \quad \left(\frac{1}{p_2}, \frac{1}{q_2} \right):= \left( \frac{d^2 - d+2}{d(d+1)}, \frac{d-1}{d+1} \right).
\end{equation*}

This condition is shown to be sufficient, at least up to an endpoint. 

\begin{thm}\label{momentthm} For $d \geq 2$ the operator $Af(x,r) := \mathcal{A}_rf(x)$ is bounded from $L^p_x(\R^d)$ to $L^q_{x,r}(\R^d\times [1,2])$ for all $(1/p,1/q) \in \mathcal{T}_d\setminus \{(1/p_1, 1/q_1)\}$.
If $(1/p,1/q) \notin \mathcal{T}_d$, then $A$ is not restricted weak-type $(p,q)$. 
\end{thm}

The proof of Theorem \ref{momentthm} proceeds by establishing a restricted weak-type inequality at the endpoint $(p_1, q_1)$. Therefore, except for the question of whether this weak-type endpoint inequality can be strengthened the theorem completely determines the $L^p$ mapping of $A$ for the given choice of curve.

The $d=2$ case of Theorem \ref{momentthm} (when the curve is also a hypersurface) is already known to hold with a strong-type inequality at the endpoint. This result essentially appears, for example, in the work of Strichartz \cite{Strichartz1977} and Schlag and Sogge \cite{Schlag1997}. Furthermore, in \cite{Strichartz1977, Schlag1997} it is observed that the critical $L^2_x-L^6_{x,r}$ inequality for dilated averages over circles is equivalent to a Stein-Tomas Fourier restriction theorem for a conic surface and connections between this theory and estimates for certain evolution equations are also discussed. In addition, the $d=2$ case follows from more recent work of Gressman \cite{Gressman2006, Gressman2013} utilising methods which are rather combinatorial in nature. The combinatorial techniques found in \cite{Gressman2006} are akin to the arguments found in the present article; both are based on earlier work of Christ \cite{Christ1998}, discussed later in the introduction. 

For $d \geq 3$ the results appear to be new and, indeed, no previous (non-trivial) partial results are known to the author. It is remarked that the connection between the theory of dilated averages, Fourier restriction and analysis of PDE appears to be confined to the hypersurface setting but nevertheless Theorem \ref{momentthm} is arguably of interest in its own right. 

Theorem \ref{momentthm} is in fact a special case of a more general result, Theorem \ref{bigthm}, described below. Indeed, rather than restricting attention to $h$, this paper considers $A$ defined with respect to any polynomial curve. In this setting the statement of the results requires some preliminary motivation and definitions.

Given an arbitrary curve $\gamma$, when investigating the mapping properties of \eqref{operator} the key consideration is curvature. One would expect $A$ is non-degenerate (in the sense that the largest possible range of estimates hold for $A$) if and only if the $d-1$ curvature functions associated to $\gamma$ are non-vanishing in the support of the density $\alpha$. This kind of phenomenon is well-known in the context of single averages and maximal functions\footnote{Although in the case of maximal functions one must also consider the position of the curve in the plane, see \cite{Iosevich1994}.} (the latter over curves in $\R^2$) and, indeed, many other operators whose definition depends on some submanifold (or family of submanifolds) of $\R^d$ (see, for instance, \cite{Christ1999, Tao2003}). One method for quantifying the relationship between the curvature of $\gamma$ and the boundedness of $A$ is to introduce a specific choice of weight $\lambda_{\gamma}$ in the definition of the operator.  In particular, $\lambda_{\gamma}$ is carefully chosen to vanish at the flat points of the curve so as to ameliorate the effect of the degeneracies. One can then hope to achieve $L^p_x - L^q_{x,r}$ boundedness for the full range of exponents corresponding to the non-degenerate case under mild hypotheses on the curve. This strategy follows the example of numerous authors (notably Drury \cite{Drury1990}, Oberlin \cite{Oberlin2002}, Dendrinos, Laghi and Wright \cite{Dendrinos2009}, Stovall \cite{Stovall2014, Stovall2010} and Dendrinos and Stovall \cite{Dendrinos}) who, in considering averages defined with respect to degenerate curves, have chosen the underlying measure in the definition of the operator to be the so-called \emph{affine arc-length} measure, described below. This measure has the desired effect of dampening any degeneracies of the curve or surface and also makes the problem both affine and parametrisation invariant. 

To make this discussion precise, define the torsion $L_{\gamma}$ of the curve to be the function
\begin{equation*}
L_{\gamma}(t) := \det (\gamma^{(1)}(t) \dots \gamma^{(d)}(t) )
\end{equation*}
where $\gamma^{(i)}$ denotes the $i$th derivative of $\gamma$, viewed as a column vector. This function vanishes precisely when any of the $d-1$ curvature functions associated to $\gamma$ vanish. The affine arc-length measure $\ud \mu_{\gamma}$ on $\gamma$ is then defined by
\begin{equation*}
\int f \,\ud \mu_{\gamma} := \int_{\R} f(\gamma(t)) \lambda_{\gamma}(t)\,\ud t
\end{equation*}
whenever $f \in C_c(\R^d)$, say, where $\lambda_{\gamma}(t) := |L_{\gamma}(t)|^{2/d(d+1)}$. In this paper the operator $A_{\gamma}$ given by convolution with dilates of this measure is studied. Explicitly,
\begin{equation}\label{Poperator}
A_{\gamma}f(x,r) := \int_{\R} f(x-r \gamma(t)) \lambda_{\gamma}(t)\,\ud t
\end{equation}
for all test functions $f$ on $\R^d$. 

The choice of weight $\lambda_{\gamma}$ is further motivated by the fact that the resulting measure exhibits both parametrisation and affine invariance (for a detailed discussion see, for example, \cite{Christc}). Consequently, for all exponents $1 \leq p,q < \infty$ satisfying the relation $1/q = 1/p - 2/d(d+1)$ any $L^p_x - L^q_{x,r}$ inequality for $A_{\gamma}$ is affine invariant in the sense that if $\gamma$ is replaced with $X \circ \gamma$ for some invertible linear transformation $X$, then the estimate still holds with the same constant. One can therefore hope to achieve estimates for $A_{\gamma}$ which are uniform over all $\gamma$ belonging to a large class of curves. A counter-example due to Sj\"olin \cite{Sjolin1974} demonstrates such uniformity is impossible if the class includes curves which exhibit an arbitrarily large number of oscillations (see also \cite{Dendrinos}). It is therefore natural to postulate uniform estimates are possible over all polynomial curves of some fixed degree, since the degree controls the number of oscillations.

\begin{thm}\label{bigthm} Let $d \geq 2$ and $P \colon \R \to \R^d$ be a polynomial curve whose components have maximum degree $n$. Then
\begin{equation*}
\|A_Pf\|_{L^q_{x,r}(\R^d \times [1,2])} \leq C_{n,d} \|f\|_{L^p_x(\R^d)}
\end{equation*}
whenever $(1/p, 1/q)$ lies in $\mathcal{T}_d \setminus\{(1/p_1, 1/q_1)\}$ and satisfies $1/q = 1/p - 2/d(d+1)$. Here the constant $C_{n,d}$ depends only on $p$, $q$, the dimension $d$ and the degree $n$ of the polynomial mapping. 
\end{thm}

Again, the theorem will follow by demonstrating a uniform restricted weak-type $(p_1, q_1)$ inequality holds. In addition, Theorem \ref{bigthm} is almost-sharp in the sense that if $(1/p, 1/q)$ does not lie in the intersection of $\mathcal{T}_d$ with the line $1/q = 1/p - 2/d(d+1)$, then such estimates do not hold with a finite constant.

It is remarked that the result of Theorem \ref{bigthm} is new for dimensions $d \geq 3$ whilst the $d=2$ case follows from a very general theorem due to Gressman \cite{Gressman2013}. Indeed, Gressman's theorem, \emph{inter alia}, establishes the hypersurface analogue of Theorem \ref{bigthm} in all dimensions, up to and including all the relevant endpoints. 

If $\R$ is replaced with a bounded interval $I$ in the definition of $A_P$, then the resulting operator (which is denoted by $A_P^{\mathrm{c}}$) is trivially seen to be bounded from $L^p_{x}$ to $L^p_{x,r}$ for all $1 \leq p \leq \infty$. Real interpolation therefore yields the following corollary.

\begin{cor} The operator $A_P^{\mathrm{c}}$ is bounded from $L^p_{x,r}$ to $L^q_x$ for all $(1/p,1/q) \in \mathcal{T}_d\setminus \{(1/p_1, 1/q_1)\}$. If $(1/p,1/q) \notin \mathcal{T}_d$ and $L_P \not\equiv 0$, then $A_P$ is not restricted weak-type $(p,q)$. 
\end{cor}

Notice, when $P = \gamma$ is the moment curve defined above, the torsion function $L_{\gamma}$ is constant and thus Theorem \ref{momentthm} is indeed a special case of Theorem \ref{bigthm}.

\begin{rmk} It is natural to ask whether the restricted weak-type $(p_1, q_1)$ endpoint can be strengthened to a strong-type estimate. This is certainly the case in dimension $d=2$ where the inequality is a consequence of the aforementioned theorem due to Gressman \cite{Gressman2013}. Furthermore, one may recover the strong-type bound for $d=2$ by combining the analysis contained within the present article with an extrapolation method due to Christ \cite{Christb} (see also \cite{Stovall2009}). It is possible that the argument can be adapted to the case where $d$ belongs to a certain congruence class modulo $3$ to (potentially) establish the strong-type bound in this situation. A more detailed discussion of the validity of the strong-type endpoint appears below in Remark \ref{strong type remark}. 
\end{rmk}

Theorems \ref{momentthm} and \ref{bigthm} belong to a growing body of works which have applied variants of the geometric and combinatorial arguments due to Christ \cite{Christ1998} to the study of operators collectively known as generalised Radon transforms, of which $A$ is an example. Essentially these operators are defined for any point $y$ belonging to $\Sigma$ an $n$-dimensional manifold by integration over a $k$-dimensional manifold $M_y$ which depends on $y$, where $k < n$ is referred to here as the \emph{dimension of the associated family}. The techniques of \cite{Christ1998} have fruitfully been applied and developed in, for instance,  \cite{Christ, Christa, Christ2002, Christ2008, Dendrinos2009, Gressman2004, Gressman2009, Stovall2014, Stovall2009, Stovall2010, Tao2003} to study the Lebesgue mapping properties of one-dimensional generalised Radon transforms $R$ which are, roughly, operators $R$ for which $R$ and its adjoint $R^*$ are both generalised Radon transforms given by integration over some family of curves. The approach has been less successful when considering $R$ which are \emph{unbalanced} in the sense that $R$ and $R^*$ are both generalised Radon transforms but the dimensions of the associated families are not equal, although it has still produced results in some specific cases, for example \cite{Erdogan2010, Gressman2006, Gressman2013}. The dilated averaging operator \eqref{operator} fits into this framework by setting $\Sigma := \R^d \times (1,2)$ and for each $(x,r) \in \Sigma$ defining $M_{(x,r)}$ to be the curve parametrised by $t \mapsto x - r \gamma(t)$. Observe that although $A$ is defined by integration over curves, the adjoint of $A$ is defined by integration over $2$-surfaces and hence the operator is unbalanced. 

The structure of the paper is as follows. In the following section the necessary conditions on $(p,q)$ for $A$ to be restricted weak type $(p,q)$ are discussed. In Section \ref{overview of refinement method} standard methods together with estimates for single averages are combined to reduce the proof of Theorem \ref{bigthm} to proving a single restricted weak-type inequality. Christ's method of refinements is also reviewed and used to establish the simple case of Theorem \ref{momentthm} when $d=3$. The remaining sections develop this method to be applicable in the general situation. 

\begin{notation} A word of explanation concerning notation is in order: throughout the paper $C$ and $c$ will be used to denote various positive constants whose value may change from line to line but will always depend only on the dimension $d$ and degree $\deg P$ of some fixed polynomial. If $X, Y \geq 0$, then the notation $X \lesssim Y$ or $Y \gtrsim X$ signifies $X \leq C Y$ and this situation is also described by ``$X$ is $O(Y)$''. In addition, $X \sim Y$ indicates $X \lesssim Y \lesssim X$. Finally, the cardinality of any finite set $B$ will be denoted by $\# B$. 

\end{notation}

\subsection*{Acknowledgement} The author wishes acknowledge his PhD supervisor, Prof. Jim Wright, for all his kind and patient guidance relating to this work. He would also like to thank both Marco Vitturi and Betsy Stovall for elucidating discussions regarding some of the references. 

\section{Necessary conditions}\label{necessity}

Suppose the operator $A$ from Theorem \ref{momentthm} satisfies a restricted weak-type $(p,q)$ inequality for some $1 \leq p,q < \infty$. Here it is shown that the exponents $p,q$ must satisfy four conditions, each corresponding to an edge of the trapezium $\mathcal{T}_d$. The first three conditions also appear in the study of the averaging operator $\mathcal{A}$ defined in the introduction and are deduced by the same reasoning. The remaining condition does not appear in the theory of single averages and here the dilation parameter plays a non-trivial r\^ole, although the arguments are only marginally different from those used to examine $\mathcal{A}$.  

\begin{figure}
\centering
\begin{tikzpicture}[
        scale=2,
        IS/.style={blue, thick},
        LM/.style={red, thick},
        axis/.style={very thick, ->, >=stealth', line join=miter},
        important line/.style={thick}, dashed line/.style={dashed, thin},
       dotted line/.style={dotted, thin},
        every node/.style={color=black},
        dot/.style={circle,fill=black,minimum size=4pt,inner sep=0pt,
            outer sep=-1pt},
    ]
    \draw[axis,-] (3.1,0) -|
                    (0,3.1) ;
    \node[below=0.2cm] at (1.5, 0) {$1/p$};
    \node[left=0.2cm] at (0, 1.5) {$1/q$};						
    \draw[important line]
            (0,0) coordinate (es) -- (3.0,3.0) coordinate (ee);
		\draw[important line]
            (0,0) coordinate (es) -- (1.0,0.5) coordinate (ee);		
    \draw[important line]
            (2.0,1.5) coordinate (es) -- (3.0,3.0) coordinate (ee);
		\draw[important line]
            (1.0,0.5) coordinate (es) -- (2.0,1.5) coordinate (ee);
 \draw[dashed line]
            (0,0) coordinate (es) -- (1.5,1.0) coordinate (ee); 
 \draw[dotted line]
            (0,3) coordinate (es) -- (3,3) coordinate (ee);
 \draw[dotted line]
            (3,0) coordinate (es) -- (3,3) coordinate (ee);

\path[] (1.0,0.5)
        node[dot] (int2) {};
 \node[below=0.2cm, right] at (1.0,0.5) {$ \big(\frac{1}{p_1}, \frac{1}{q_1}\big)$};
\path[] (2.0,1.5)
        node[dot] (int2) {};
 \node[below=0.2cm, right] at  (2.0,1.5) {$ \big(\frac{1}{p_2}, \frac{1}{q_2}\big)$};
\path[] (1.5,1)
        node[dot] (int2) {};
 \node[below=0.2cm, right] at  (1.5,1) {$ \big(\frac{1}{q_2'}, \frac{1}{p_2'}\big)$};				
\path[] (0,0)
        node[dot,label=below:{$(0,0)$}] (int2) {};
\path[] (3,3)
        node[dot,label=right:{$(1,1)$}] (int2) {};
	
\end{tikzpicture}
\caption{The operator $A$ of Theorem \ref{momentthm} is restricted weak-type $(p,q)$ if and only if $(1/p, 1/q)$ belongs to the (closed set bounded by the) bold trapezium. For comparison, the single average $\mathcal{A}$ defined in the introduction is restricted weak-type $(p,q)$ if and only if $(1/p, 1/q)$ belongs to the smaller trapezium formed by introducing the dashed line.} \label{fig}
\end{figure}
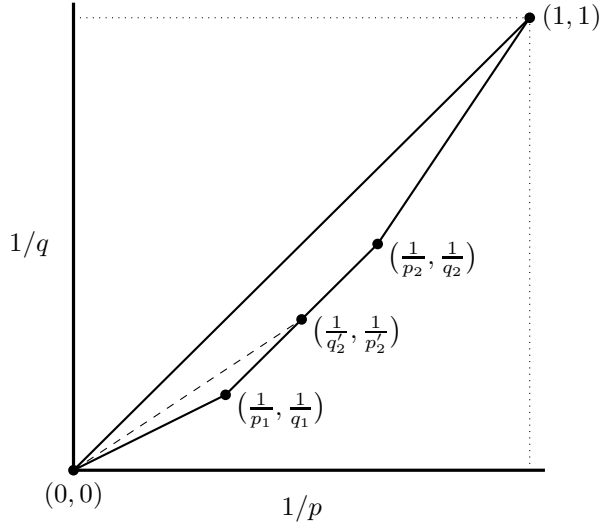

\begin{proof}[Proof (of Theorem \ref{momentthm}, necessity)]

To begin, a slight modification of a general theorem of H\"ormander \cite{Hormander1960} implies $p \leq q$.

For the second condition, let $R(\delta) : = \prod_{j=1}^d [-\delta^j,\delta^j]$ and note that
\begin{equation*}
A \chi_{R(\delta)}(x,r) = |\{ t \in [0,1] : x - r\gamma (t) \in R(\delta) \}|.
\end{equation*}
If $x \in (1/2) R(\delta)$, then whenever $t \in [0, \delta/4]$ it follows that
\begin{equation*}
|x_j -rt^j| \leq \delta^j/2 + 2 (\delta/4)^j \leq \delta^j \qquad \textrm{for $j = 1, \dots, d$}
\end{equation*}
and therefore
\begin{equation*}
A \chi_{R(\delta)}(x,r) \geq \frac{\delta}{4} \chi_{(1/2)R(\delta)}(x).
\end{equation*}
Consequently, applying the hypothesised restricted weak-type estimate,
\begin{equation*}
|R(\delta)| \lesssim \left|\left\{ (x,r) \in \R^d \times [1,2] : A \chi_{R(\delta)}(x,r) > \delta/8 \right\}\right| \lesssim \left( \frac{1}{\delta}|R(\delta)|^{1/p}\right)^{q}.
\end{equation*}
Observe $|R(\delta)| \sim_d \delta^{d(d+1)/2}$ and so the preceding inequality implies
\begin{equation*}
\delta^{d(d+1)/(2q)} \lesssim \delta^{d(d+1)/(2p) - 1} \qquad \textrm{for all $0 < \delta  < 1$.}
\end{equation*}
 The exponents $(p,q)$ must therefore satisfy the relation
\begin{equation*}
\frac{1}{q} \geq \frac{1}{p} - \frac{2}{d(d+1)}.
\end{equation*} 
The third condition is established by testing $A$ on $\chi_{B(\delta)}$, the characteristic function of a ball $B(\delta) \subset \R^d$ of radius $0 < \delta < 1$, centred at the origin. It is easy to see
\begin{equation}\label{nec2}
A \chi_{B(\delta)}(x,r) \gtrsim \delta \chi_{\mathcal{N}_r(\delta)}(x)
\end{equation}
where $\mathcal{N}_r(\delta)$ is a $\delta/3$-neighbourhood of the $r$-dilate of the moment curve; that is, the set of all points $x \in \R^d$ for which $|x- r\gamma(t_0)| < \delta/3$ for some $t_0 \in [0,1]$. The hypothesised restricted weak-type estimate together with \eqref{nec2} imply
\begin{eqnarray*}
\left|\left\{(x,r) \in \R^d \times [1,2] : x \in \mathcal{N}_r(\delta)\right\}\right| &\leq& \left|\left\{ (x,r) \in \R^d \times [1,2] : A \chi_{B(\delta)}(x,r) > C \delta \right\}\right| \\
&\lesssim& \left( \frac{1}{\delta}|B(\delta)|^{1/p}\right)^{q}.
\end{eqnarray*}
Observe $|B(\delta)| \sim_d \delta^d$ whilst $|\mathcal{N}_r(\delta)| \gtrsim \delta^{d-1}$ for all $r \in [1,2]$ and so the preceding inequality implies
\begin{equation*}
\delta^{(d-1)/q} \lesssim \delta^{d/p - 1}\qquad \textrm{for all $0 < \delta  < 1$.}
\end{equation*}
Thus the exponents must satisfy the relation
\begin{equation*}
\frac{1}{q} \geq \frac{d}{d-1} \frac{1}{p} - \frac{1}{d-1}.
\end{equation*} 

The final condition on $(1/p, 1/q)$ is deduced by considering the adjoint $A^*$ of $A$. A simple computation yields
\begin{equation*}
A^*g(x) = \int_1^2 \int_0^1 g(x + r\gamma(t), r)\,\ud t \ud r
\end{equation*}
for suitable functions $g$ defined on $\R^d \times [1,2]$. The hypothesis on $(p,q)$ is equivalent to the assumption that $A^*$ is restricted weak-type $(q',p')$. For $B(\delta)$ as above, let $F(\delta)$ denote the set $B(\delta) \times [1, 1+c\delta]$ for some small constant $c$. Observe
\begin{equation*}
A^* \chi_{F(\delta)}(x) \gtrsim \delta^2 \chi_{\mathcal{N}_1(\delta)}(-x)
\end{equation*}
where $\mathcal{N}_1(\delta)$ is as defined above. Therefore,
\begin{equation*}
|\mathcal{N}_1(\delta)| \leq \left|\left\{ x \in \R^d : A^* \chi_{F(\delta)}(x) \gtrsim \delta^2 \right\}\right| \lesssim \left( \frac{1}{\delta^{2}}|F(\delta)|^{1/q'}\right)^{p'}.
\end{equation*}
Finally, $|F(\delta)| \sim_d \delta^{d+1}$ whilst $|\mathcal{N}_1(\delta)| \gtrsim \delta^{d-1}$ and so the preceding inequality implies
\begin{equation*}
\delta^{(d-1)/p'} \lesssim \delta^{(d+1)/q' - 2} \qquad \textrm{for all $0 < \delta  < 1$.}
\end{equation*}
It follows that the exponents must satisfy the relation $(d+1)/q' - 2 \leq (d-1)/p'$ which can be rewritten as
\begin{equation*}
\frac{1}{q} \geq \frac{d-1}{d+1} \frac{1}{p}.
\end{equation*}

\end{proof}

\section{An overview of the refinement method}\label{overview of refinement method}

It remains to show the conditions on $(p,q)$ described in Theorem \ref{bigthm} are sufficient to ensure $A_P$ satisfies a type $(p,q)$ inequality with the desired uniformity. Real interpolation immediately reduces matters to establishing a uniform restricted weak-type $(p_1, q_1)$ and strong type $(p_2, q_2)$ estimate for $A_P$. The latter is easily dealt with by appealing to the existing literature. Indeed, a theorem of Stovall \cite{Stovall2010} implies the estimate
\begin{equation}\label{stovallest}
\|A_Pf( \,\cdot\, , r) \|_{L^{q_2}_x(\R^d)} \lesssim \|f\|_{L^{p_2}_x(\R^d)}
\end{equation}
holds for all $r \in [1,2]$. Taking $L^{q_2}_r([1,2])$-norms of both sides of \eqref{stovallest} yields the uniform type $(p_2, q_2)$ inequality for $A_P$ and the proof of Theorem \ref{bigthm} is therefore reduced to establishing the following Proposition.

\begin{prop}\label{weakthm} For $d \geq 2$ the inequality
\begin{equation}\label{weak}
\langle A_P\chi_E, \chi_F \rangle \lesssim |E|^{1/d}|F|^{(d^2+1)/d(d+1)}
\end{equation}
is valid for all pairs of Borel sets $E \subset \R^d$ and $F \subset \R^d \times [1,2]$ of finite Lebesgue measure. 
\end{prop}

The proof of Proposition \ref{weakthm} will utilise the geometric and combinatorial techniques introduced by Christ in \cite{Christ1998}, which were briefly discussed in the introduction. Collectively these techniques are referred to as the method of refinements. In this section the rudiments of the method are reviewed. It is instructive to consider the proof of the analogue of Proposition \ref{weakthm} in three dimensions ($d=3$) for the operator $A$ from the statement of Theorem \ref{momentthm}. In this situation the arguments are extremely simple and only a crude version of the refinement procedure is required.

Let $E$ and $F$ denote fixed sets satisfying the hypotheses of Proposition \ref{weakthm} for $d=3$. Assume, without loss of generality, that $\langle A\chi_E, \chi_F \rangle \neq 0$ where $A$ is the operator from Theorem \ref{momentthm}. One wishes to establish the inequality
\begin{equation*}
\langle A\chi_E, \chi_F \rangle \lesssim |E|^{1/3}|F|^{5/6},
\end{equation*}
from which Theorem \ref{momentthm} follows for the case $d=3$. Defining constants $\alpha$ and $\beta$ by the equation $\langle A\chi_E, \chi_F \rangle = \alpha|F| = \beta|E|$, one may rewrite the preceding inequality as a lower bound on the measure of $E$; explicitly,
\begin{equation}\label{momentequivweake}
|E| \gtrsim \alpha^{6} (\beta/\alpha).
\end{equation}

The basic idea behind Christ's method is to attempt to prove \eqref{momentequivweake} by using iterates of $A$ and $A^*$ to construct a natural parameter set $\Omega \subset \R^3$ and parametrising function $\Phi : \Omega \rightarrow E$ with a number of special properties. First of all, $\Phi$ must have bounded multiplicity so, by applying the change of variables formula,
\begin{equation*}
|E| \gtrsim \int_{\Omega} |J_{\Phi}(t)| \,\ud t 
\end{equation*}
where $J_{\Phi}$ denotes the Jacobian of $\Phi$. It then remains to bounded this integral from below by some expression in terms of $\alpha$ and $\beta$, which is possible provided that the parametrisation has been carefully constructed.

Following \cite{Christ1998}, define
\begin{eqnarray*}
F_1 &:=& \big\{ (x,r) \in F : A \chi_E(x,r) > \alpha / 2 \big\}, \\
E_1 &:=& \big\{ y \in E : A^* \chi_{F_1}(y)  > \beta / 4 \big\}.
\end{eqnarray*}
It is not difficult to see the assumptions on $E$ and $F$ imply $\langle A\chi_{E_1}, \chi_{F_1} \rangle \neq 0$ and therefore $E_1$ is non-empty. Fix $y_0 \in E_1$ and define a map $\Phi_1 \colon [1,2] \times [0,1] \to \R^3 \times [1,2]$ by
\begin{equation}\label{moment Phi1}
\Phi_1(r_1,t_1) := \left(\begin{array}{c}
y_0 + r_1h(t_1) \\ 
r_1 
\end{array} \right).
\end{equation}
Note that the set
\begin{equation*}
\Omega_1:= \left\{ (r_1, t_1) \in [1,2] \times [0,1] : \Phi_1(r_1,t_1) \in F_1 \right\}
\end{equation*}
satisfies $|\Omega_1| > \beta / 4$. Similarly, define a map $\Phi_2 \colon [1,2] \times [0,1]^2 \to \R^3$ by
\begin{equation}\label{moment Phi2}
\Phi_2(r_1, t_1, t_2) := y_0 + r_1h(t_1) - r_1h(t_2)
\end{equation}
and observe for each $(r_1, t_1) \in \Omega_1$ the set
\begin{equation*}
\Omega_2(r_1, t_1):= \left\{  t_2 \in [0,1] : \Phi_2(r_1, t_1, t_2) \in E \right\}
\end{equation*}
satisfies $|\Omega_2(r_1, t_1)| > \alpha / 2$. Finally, define the structured set
\begin{equation*}
\Omega_2 := \big\{(r_1, t_1, t_2) \in [1,2] \times [0,1]^2 : (r_1, t_1) \in \Omega_1 \textrm{ and } t_2 \in \Omega_2(r_1, t_1)  \big\}.
\end{equation*}
Now, $\Omega := \Omega_2 \subset \R^3$ is the parameter set alluded to above and $\Phi := \Phi_2 |_{\Omega} : \Omega \rightarrow E$ the parametrising function. Observe $\Phi$ is well-defined by the preceding observations and the polynomial nature of this map ensures it has almost everywhere bounded multiplicity.\footnote{That is, for almost every $x \in \R^d$ the cardinality of the pre-image $\Phi^{-1}(\{x\})$ is no greater than some fixed (finite) constant. For further details see Lemma \ref{multiplicity} below.} The absolute value of the Jacobian $J_{\Phi}(r_1, t_1, t_2)$ of $\Phi$ may be expressed as
\begin{equation*}
r_1^2 \left| \det \left(\begin{array}{ccc}
1 & 1 & t_2 - t_1 \\
2t_1 & 2t_2 & t_2^2 - t_1^2 \\
3t_1^2 & 3t_2^2 & t_2^3 - t_1^3 \end{array}\right) \right| = 6r_1^2 \bigg| \int_{t_1}^{t_2} V(t_1, t_2, x) \,\ud x \bigg|
\end{equation*}
where $V(x_1, \dots, x_m) := \prod_{1\leq i<j \leq m} (x_j - x_i)$ denotes, and will always denote, the $m$-variable Vandermonde polynomial. The sign of $V(t_1, t_2, x)$ does not change as $x$ varies between $t_1$ and $t_2$ and so modulus signs can be placed inside the integral in the above expression. Thus, 
\begin{equation*}
|J_{\Phi}(r_1, t_1, t_2)| \gtrsim |t_1 - t_2| \int_{t_1}^{t_2} |x-t_1||t_2-x| \,\ud x \gtrsim |t_1 - t_2|^4,
\end{equation*}
where the last inequality may easily be deduced by removing a $|t_1-t_2|/8$-neighbourhood of the endpoints $\{t_1,t_2\}$ from the domain of integration. 
Consequently, by applying the change of variables formula, 
\begin{equation*}
|E| \gtrsim \int_{\Omega} |J_{\Phi}(r_1, t_1, t_2)| \,  \ud t_2 \ud r_1 \ud t_1  \gtrsim \int_{\Omega_1}\int_{\Omega_2(r_1, t_1)} |t_1 - t_2|^4 \,  \ud t_2 \ud r_1 \ud t_1.
\end{equation*}
For each $(t_1,r_1) \in \Omega_1$ define $\widetilde{\Omega}_2(r_1, t_1) := \Omega_2(r_1, t_1)\setminus (t_1 + c\alpha, t_1 - c\alpha)$ for a suitably small constant $c$, chosen so that $|\widetilde{\Omega}_2(r_1, t_1)| \gtrsim \alpha$. Hence, 
\begin{equation*}
|E| \gtrsim \int_{\Omega_1}\int_{\widetilde{\Omega}_2(r_1, t_1)} |t_1 - t_2|^4 \,  \ud t_2 \ud r_1 \ud t_1 \gtrsim \alpha^4 \int_{\Omega_1}\int_{\widetilde{\Omega}_2(r_1, t_1)} \,  \ud t_2 \ud r_1 \ud t_1 \gtrsim  \alpha^5 \beta,
\end{equation*}
and this concludes the proof of \eqref{momentequivweake} and thereby establishes Theorem \ref{momentthm} for $d=3$.

The remainder of the paper will develop this elementary argument in order to prove Proposition \ref{weakthm} in any dimension $d$ and for any polynomial curve $P$.

\section{The polynomial decomposition theorem of Dendrinos and Wright}\label{decomposition 1}

The refinement method essentially reduces the problem of establishing the restricted weak-type inequality \eqref{weak} from Proposition \ref{weakthm} to estimating a Jacobian determinant associated with a certain naturally arising change of variables. In the case of the moment curve this Jacobian takes a particularly simple form involving a Vandermonde polynomial $V(t)$. For a general polynomial curve $P \colon \R \to \R^d$ one is led to consider expressions of the form
\begin{equation}\label{JP}
J_P(t) := \det(P'(t_1) \dots P'(t_d)) 
\end{equation}
for $t = (t_1, \dots, t_d) \in \R^d$.\footnote{For the moment curve $h(t) := (t, t^2, \dots, t^d)$, one immediately observes that $J_h(t) = cV(t)$.} The multivariate polynomial $J_P$ can be effectively estimated by comparing it with the Vandermonde polynomial and a certain geometric quantity expressed in terms of the torsion function (whose definition is recalled below). This leads to what is referred to here (and in \cite{Dendrinos2010}) as a \emph{geometric inequality} for $J_P$. It is often the case that such a comparison is not possible globally; however, an important theorem due to Dendrinos and Wright \cite{Dendrinos2010} demonstrates the existence of a decomposition of the real line into a bounded number of intervals, $\R = \bigcup_{m=1}^C \overline{I_m}$, such that such a geometric inequality holds on each constituent interval $I_m$. Furthermore, the torsion function has a particularly simple form when restricted to an $I_m$: it is comparable to a centred monomial. Restricting the analysis to an interval arising from the Dendrinos-Wright decomposition therefore significantly simplifies the situation and allows for an effective estimation of the Jacobian $J_P$.

In order to state the decomposition lemma, recall the torsion of the curve $P$ is defined to be the polynomial function
\begin{equation*}
L_P(t) := \det(P^{(1)}(t) \dots P^{(d)}(t))
\end{equation*}
where $P^{(i)}$ denotes the $i$th derivative of $P$.

\begin{thm}[Dendrinos and Wright \cite{Dendrinos2010}]\label{Dendrinos Wright theorem} Let $P \colon \R \to \R^d$ be a polynomial curve of degree $n$ such that $L_P \not\equiv 0$. There exists an integer $C = C_{d,n}$ and a decomposition $\R = \bigcup_{m=1}^{C} \overline{I_m}$ where the $I_m$ are pairwise disjoint open intervals with the following properties:
\begin{enumerate}[1)]
\item Whenever $\mathbf{t} = (t_1, \dots, t_d) \in I_m^d$ the geometric inequality 
\begin{equation*}
|J_P(\mathbf{t})| \gtrsim \prod_{i=1}^d|L_P(t_i)|^{1/d} |V(\mathbf{t})|
\end{equation*}
holds.
\item For every $1 \leq m \leq C$ there exists a positive constant $D_{m}$, a non-negative integer $K_{m} \lesssim 1$ and a real number $b_{m} \in \R \setminus I_m$ such that
\begin{equation*}
|L_{P}(t)| \sim D_{m}|t - b_m|^{K_{m}} \qquad \textrm{for all $t \in I_m$.}
\end{equation*}
\end{enumerate}
\end{thm}

Theorem \ref{Dendrinos Wright theorem} originally appeared in \cite{Dendrinos2010} where it was used to study Fourier restriction operators associated to polynomial curves (see \cite{Stovall} for further developments in this direction). Concurrently, Dendrinos, Laghi and Wright \cite{Dendrinos2009} applied the decomposition to establish uniform estimates for convolution with affine arc-length on polynomial curves in low dimensions; their results were subsequently extended to all dimensions by Stovall \cite{Stovall2010}. Many of the methods of this paper are based on those found in \cite{Dendrinos2009, Stovall2010}.

Fixing a polynomial $P$ for which $L_P \not\equiv 0$, to prove Proposition \ref{weakthm} it suffices to establish the analogous uniform restricted weak-type inequalities for the local operators
\begin{equation*}
A^c_P(x,r) := \int_I f(x - r P(t)) \, \lambda_P(t)\ud t
\end{equation*} 
where $I$ is any bounded interval. Furthermore, one may assume $I$ lies completely within one of the intervals $I_m$ produced by the decomposition (indeed, $A^c_P$ can always be expressed as a sum of a bounded number of operators of the same form for which this property holds). Observing the translation, reflection and scaling invariance of the problem, one may assume $D_{m} = 1$, $b_m = 0$ and $I \subset (0, \infty)$ with $|I| = 1$ without any loss of generality. Similar reductions were made in \cite{Stovall2010} where further details can be found. Notice under these hypotheses, $|L_P(t)| \sim t^K$ uniformly on $I$ for some non-negative integer $K \lesssim 1$.

Henceforth $A$ will denote the operator defined by
\begin{equation}\label{reduced operator}
Af(x, r) := \int_I f(x - r P(t)) \,\ud \mu_P(t)
\end{equation}
where $\mu_P$ is now the weighted measure $\ud \mu_P(t) := \lambda_P(t)\ud t$; $\lambda_P$ is redefined as $\lambda_P(t):= t^{2K/d(d+1)}$ and the integer $K$ and interval $I$ satisfy the above properties. It remains to prove the analogue of the restricted weak-type inequality \eqref{weak} from Proposition \ref{weakthm} for this operator.

To close this section it is remarked that Stovall \cite{Stovall2010} established an upper bound for certain derivatives of $J_P$ on the set $I^d$ in terms of $J_P$ itself. This estimate will be of use in the forthcoming analysis and is recorded presently for the reader's convenience.

\begin{prop}[Stovall \cite{Stovall2010}]\label{Stovall's observation} Let $S \subseteq \{1, \dots, d\}$ be a non-empty set of indices. Whenever $t = (t_1, \dots, t_d) \in I^d$, one has the estimate
\begin{equation*}
\bigg| \prod_{j \in S} \frac{\partial}{\partial t_j} J_P(t) \bigg| \lesssim \sum_{T \subseteq S} \sum_{u, \epsilon} \bigg(\prod_{{j} \in S\setminus T} t_{j}^{-1}\bigg)\bigg( \prod_{{j} \in T} t_{j}^{-\epsilon(j)}|t_{j} - t_{u({j})}|^{\epsilon(j) - 1}\bigg)|J_P(t)|
\end{equation*}
where the outer sum is over all subsets $T$ of $S$ and the inner sum is over all functions $u \colon T \to \{1, \dots, d\}$ with the property $u(j) \neq j$ for all $j \in T$ and all $\epsilon \colon T \to \{0,1\}$.  
\end{prop}

\section{Parameter towers}

Having made the reductions of the previous section, fix Borel sets $E \subseteq \R^d$ and $F \subseteq \R^d \times [1,2]$ of finite Lebesgue measure such that $\langle A \chi_E\,,\, \chi_F\rangle \neq 0$ where $A$ is of the special form described in \eqref{reduced operator}. As in Section \ref{overview of refinement method}, the quantities
\begin{equation*}
\alpha := \frac{1}{|F|} \langle A\chi_E, \chi_F \rangle \quad \textrm{and} \quad \beta := \frac{1}{|E|} \langle \chi_E, A^*\chi_F \rangle
\end{equation*}
play a dominant r\^ole in the analysis. Indeed, by some simple algebra the inequality \eqref{weak} can be restated in terms of $\alpha$ and $\beta$ as either
\begin{equation}\label{equivweake}
|E| \gtrsim \alpha^{d(d+1)/2} (\beta/\alpha)^{(d-1)/2}
\end{equation}
or
\begin{equation}\label{equivweakf}
|F| \gtrsim \alpha^{d(d+1)/2} (\beta/\alpha)^{(d+1)/2}.
\end{equation}
The proof will proceed by attempting to establish either one of these estimates by applying a variant of the refinement procedure described earlier. In view of the $L^{p_2}_x - L^{q_2}_{x,r}$ estimate established in Section \ref{overview of refinement method}, henceforth it is assumed without loss of generality that $\alpha > \beta$. Indeed, the restricted weak-type $(p_2, q_2)$ inequality implies
\begin{equation*}
|E| \gtrsim \alpha^{d(d+1)/2} (\beta/\alpha)^{d}
\end{equation*}
from which \eqref{equivweake} follows in the case $\alpha \leq \beta$.

As in Section \ref{overview of refinement method}, either \eqref{equivweake} or \eqref{equivweakf} will be established by constructing suitable parameter domain $\Omega$ and parametrising function $\Phi$ where $\Omega$ is some structured set. In this section the basic structure of such a domain $\Omega$ is described. 

Consider a collection $\{\Omega_j\}_{j=1}^D$ of Borel measurable sets either of the form\footnote{Here $\ceil{x} := \min\{ n \in \N : n \geq x\}$ and $\floor{x} := \max\{ n \in \N : n \leq x\}$ for any $x \in \R$.}
\begin{equation}\label{floortower}
\Omega_j \subseteq [1,2]^{\floor{j/2}} \times I^j \qquad \textrm{for $j = 1, \dots, D$}
\end{equation}
or 
\begin{equation}\label{ceiltower}
\Omega_j \subseteq [1,2]^{\ceil{j/2}} \times I^j \qquad \textrm{for $j = 1, \dots, D$}.
\end{equation}
In order to be concise it is useful to let $\bracket{x}$ ambiguously denote either $\ceil{x}$ or $\floor{x}$ for any $x \in \R$, where it is understood the notation is consistent within any given equation. Thus \eqref{floortower} and \eqref{ceiltower} are considered simultaneously by writing
\begin{equation*}
\Omega_j \subseteq [1,2]^{\bracket{j/2}} \times I^j \qquad \textrm{for $j = 1, \dots, D$.}
\end{equation*}
Assume each $\Omega_j$ has positive $(j + \bracket{j/2})$-dimensional measure. The following definitions, which borrow terminology from \cite{Christ, Christa}, are fundamental in what follows.

\begin{dfn} 
\begin{enumerate}[i)]
\item A collection $\{\Omega_j\}_{j=1}^D$ of the above form is a (parameter) tower of height $D \in \N$ if for any $1 < j \leq D$ and $r_1, \dots, r_{\bracket{j/2}} \in [1,2]$ and $t_1, \dots, t_j \in I$ the following holds:
\begin{equation*}
(\mathbf{r}_j, \mathbf{t}_j) \in \Omega_j \Rightarrow (\mathbf{r}_{j-1}, \mathbf{t}_{j-1}) \in \Omega_{j-1}
\end{equation*}
where $\mathbf{r}_k := (r_1, \dots, r_{\bracket{k/2}})$ and $\mathbf{t}_k := (t_1, \dots, t_k)$ for $k = j-1, j$. 
\item If a tower is described as ``type 1'' (respectively, ``type 2'') this indicates the constituent sets are of the form described in \eqref{floortower} (respectively, \eqref{ceiltower}). Thus, when considering type 1 (respectively, type 2) towers the symbol $\bracket{x}$ is interpreted as $\floor{x}$ (respectively, $\ceil{x}$) for any $x \in \R$. 

\item Given a type 1 (respectively, type 2) tower $\{\Omega_j\}_{j=1}^D$, fix $1< j \leq D$. For each $(\mathbf{r}_{j-1}, \mathbf{t}_{j-1}) \in \Omega_{j-1}$ define the associated fibre $\Omega_j(\mathbf{r}_{j-1}, \mathbf{t}_{j-1})$ to be the set
\begin{equation*}
\left\{\begin{array}{ll}
\big\{ t_j \in I : (\mathbf{r}_j, \mathbf{t}_j) \in \Omega_j \big\} 
 & \textrm{if $j$ is odd (respectively even)}\\
&\\
\left\{ (r_{\bracket{j/2}}, t_j) \in [1,2] \times I : (\mathbf{r}_j, \mathbf{t}_j) \in \Omega_j \right\} & \textrm{if $j$ is even (respectively odd).}
\end{array}\right.
\end{equation*}
\end{enumerate}
\end{dfn}

\begin{rmk} A type $j$ tower is characterised by the property that the initial set $\Omega_1$ is $j$-dimensional, for $j=1,2$. 
\end{rmk}

For example, the collection $\{\Omega_j\}_{j=1}^2$ defined in Section \ref{overview of refinement method} constitutes a type 2 tower. In what follows, type 1 towers will be of primary interest. The elements of the various levels of a type 1 tower are typically denoted using the following notation:
\begin{equation*}
\mathbf{t}_1 = t_1 \in \Omega_1, \quad (\mathbf{r}_2, \mathbf{t}_2) = (r_1, t_1,t_2) \in \Omega_2, \quad (\mathbf{r}_3, \mathbf{t}_3) = (r_1, t_1, t_2, t_3) \in \Omega_3,
\end{equation*}
\begin{equation*}
  (\mathbf{r}_4, \mathbf{t}_4) = (r_1, r_2, t_1, t_2, t_3, t_4) \in \Omega_4,\quad (\mathbf{r}_5, \mathbf{t}_5) = (r_1, r_2, t_1, t_2, t_3, t_4, t_5) \in \Omega_5, \dots .
\end{equation*}

Recall that certain mappings $\Phi_1$ and $\Phi_2$, defined in \eqref{moment Phi1} and \eqref{moment Phi2}, were associated to the tower constructed in Section \ref{overview of refinement method}. Presently the analogues of these mappings in the general situation are discussed. First of all one associates to every $(x_0, r_0) \in \R^d \times [1,2]$ and $y_0 \in \R^d$ a family of functions.
\begin{enumerate}[i)]
\item Given $(x_0, r_0) \in \R^d \times [1,2]$ define the functions $\Psi_j(x_0, r_0;\,\cdot\,) : [1,2]^{\floor{j/2}} \times I^j \to \R^d$  by
\begin{equation}\label{parf}
\Psi_j(x_0, r_0;\mathbf{r}_j, \mathbf{t}_j) = x_0 + \sum_{k = 1}^j (-1)^{k} r_{\floor{k/2}} P(t_k).
\end{equation}
for all $\mathbf{r}_j = (r_1, \dots, r_{\floor{j/2}}) \in [1,2]^{\floor{j/2}}$ and $\mathbf{t}_j = (t_1, \dots, t_j) \in I^j$.
\item Given $y_0 \in \R^d$ define the functions $\Psi_j(y_0;\,\cdot\,) : [1,2]^{\ceil{j/2}} \times I^j \to \R^d$ by
\begin{equation}\label{pare}
\Psi_j(y_0; \mathbf{r}_j, \mathbf{t}_j) = y_0 + \sum_{k = 1}^j (-1)^{k+1} r_{\ceil{k/2}} P(t_k)
\end{equation}
for all $\mathbf{r}_j = (r_1, \dots, r_{\ceil{j/2}}) \in [1,2]^{\ceil{j/2}}$ and $\mathbf{t}_j = (t_1, \dots, t_j) \in I^j$. 
\end{enumerate}

To any tower one associates a family of mappings on the constituent sets, defined in terms of the $\Psi_j$ functions.

\begin{dfn}\label{associated mappings} Suppose $\{\Omega_j\}_{j=1}^D$ is a type 1 (respectively, type 2) tower and fix some $z_0 = (x_0, r_0) \in \R^d \times [1,2]$ (respectively, $z_0 = y_0 \in \R^d$). The family of mappings $\{\Phi_j\}_{j=1}^D$ associated to these objects is defined as follows:
\begin{enumerate}[i)]
\item For $1 \leq j \leq D$ odd (respectively, even) let $\Phi_j : \Omega_j \rightarrow \R^d$ denote the map
\begin{equation*}
\Phi_j(\mathbf{r}_j, \mathbf{t}_j) := \Psi_j(z_0; \mathbf{r}_j, \mathbf{t}_j).
\end{equation*}
\item For $1 \leq j \leq D$ even (respectively, odd) let $\Phi_j : \Omega_j \rightarrow \R^d \times [1,2]$ denote the map
\begin{equation*}
\Phi_j(\mathbf{r}_j, \mathbf{t}_j) := \left( \begin{array}{c}
\Psi_k(z_0, \mathbf{r}_j, \mathbf{t}_j)\\
r_{\bracket{j/2}} \end{array}\right).
\end{equation*}
\end{enumerate}
\end{dfn}

Referring back to the simple case discussed earlier, (appropriate restrictions of) the functions defined in \eqref{moment Phi1} and \eqref{moment Phi2} are easily seen to constitute the family associated to the point $y_0$ and tower $\{\Omega_j\}_{j=1}^2$ constructed in Section \ref{overview of refinement method}. 

For notational convenience define the following quantity
\begin{equation}\label{kappa}
\kappa := \frac{d(d+1)}{2K + d(d+1)}.
\end{equation}
Recalling the definition of $\mu_P$ from \eqref{reduced operator}, it is also useful to let $\nu_P$ denote the measure given by the product of Lebesgue measure on $[1,2]$ with $\mu_P$. Hence, for any Borel set $R \subseteq [1,2] \times I$,
\begin{equation*}
\nu_P(R) = \int_1^2 \int_I \chi_R(r,t)\,\lambda_P(t)\ud t \ud r.
\end{equation*}
Initially the following lemma is used to construct a suitable parameter tower.

\begin{lem}\label{dendrinosstovall} There exists a point $(x_0, r_0) \in F$ and a type 1 tower $\{\Omega_j\}_{j = 1}^{d+1}$ with the following properties:
\begin{enumerate}[1)]
\item Whenever $(\mathbf{r}_j, \mathbf{t}_j) \in \Omega_j$ it follows that
\begin{equation*}
\alpha^{\kappa} = \max\{\alpha, \beta\}^{\kappa} \lesssim t_1 < t_2 < \dots < t_j.
\end{equation*}
\item For $1 \leq j \leq d+1$ odd:
\begin{enumerate}[i)]
\item $\Phi_j(\Omega_j) \subseteq E$;
\item $\mu_P(\Omega_1) \gtrsim \alpha$ and if $j > 1$, then $\mu_P\left(\Omega_j(\mathbf{r}_{j-1}, \mathbf{t}_{j-1})\right) \gtrsim \alpha$ whenever $(\mathbf{r}_{j-1}, \mathbf{t}_{j-1}) \in \Omega_{j-1}$;
\item If $j > 1$ and $(\mathbf{r}_j, \mathbf{t}_j) \in \Omega_j$, then 
\begin{equation*}
\int_{t_{j-1}}^{t_j} \,\lambda_P(t)\ud t \gtrsim \alpha.
\end{equation*}
\end{enumerate}
\item  For $1 < j \leq d+1$ even:
\begin{enumerate}[i)]
\item $\Phi_j(\Omega_j) \subseteq F$;
\item $\nu_P \left(\Omega_j(\mathbf{r}_{j-1}, \mathbf{t}_{j-1})\right) \gtrsim \beta$ whenever $(\mathbf{r}_{j-1}, \mathbf{t}_{j-1}) \in \Omega_{j-1}$;
\item If $(\mathbf{r}_j, \mathbf{t}_j) \in \Omega_j$, then 
\begin{equation*}
\int_{t_{j-1}}^{t_j} \,\lambda_P(t)\ud t \gtrsim \beta.
\end{equation*}
\end{enumerate}
\end{enumerate}
\end{lem}

Lemma \ref{dendrinosstovall} is a slight modification of a recent result due to Dendrinos and Stovall \cite{Dendrinos}, based on a fundamental construction due to Christ \cite{Christ1998}. Rather than present a proof of Lemma \ref{dendrinosstovall} a stronger statement, Lemma \ref{towerlem}, is established below.

To conclude this section it is noted that a tower admitting all the properties described in the previous lemma automatically satisfies a certain separation condition. This observation was also used in \cite{Dendrinos}.

\begin{cor}\label{separation} Let $\{\Omega_j\}_{j=1}^{d+1}$ be a tower with all the properties described in Lemma \ref{dendrinosstovall}.
\begin{enumerate}[i)]
\item Suppose $1 < j \leq d+1$ is odd. Then, for all $(\mathbf{r}_j, \mathbf{t}_j) \in \Omega_j$ it follows that
\begin{equation*}
t_j - t_i \gtrsim \alpha t_i^{-2K/d(d+1)} \qquad \textrm{for $1 \leq i \leq j-1$.}
\end{equation*}
\item Suppose $1 < j \leq d+1$ is even. Then, for all $(\mathbf{r}_j, \mathbf{t}_j) \in \Omega_j$ it follows that
\begin{equation*}
t_j - t_{j-1} \gtrsim \beta t_{j-1}^{-2K/d(d+1)}; \qquad
t_j - t_{i} \gtrsim \alpha t_i^{-2K/d(d+1)} \quad \textrm{for $1 \leq i \leq j-2$.}
\end{equation*}
\end{enumerate}
\end{cor}

\begin{proof} Let $1 < j \leq n+1$ be either odd or even and $1 \leq i \leq j-1$. If $j$ is even, then further suppose $i\leq j-2$. For $(\mathbf{r}_j, \mathbf{t}_j) \in \Omega_j$, properties 1) and 2) iii) of the construction ensure there exists some $t_{i} < s_{i} < t_{j}$ for which
\begin{equation*}
\int_{t_{i}}^{s_{i}} \lambda_P(t)\,\ud t \sim \alpha.
\end{equation*}
Consequently,
\begin{equation*}
s_{i}^{1/\kappa} \sim \int_0^{s_{i}} \lambda_P(t)\,\ud t = \int_0^{t_{i}} \lambda_P(t)\,\ud t + \int_{t_{i}}^{s_{i}} \lambda_P(t)\,\ud t \sim t_{i}^{1/\kappa} + \alpha
\end{equation*}
and, since $\alpha \lesssim t_{i}^{1/\kappa}$ holds by property 1), one concludes that $s_{i} \sim t_{i}$. Whence,
\begin{equation*}
|t_j - t_i| \geq |s_{i} - t_{i}| \gtrsim  \bigg(\int_{t_{i}}^{s_{i}} \lambda_P(t)\,\ud t\bigg)t_{i}^{-2K/d(d+1)} \gtrsim \alpha t_{i}^{-2K/d(d+1)}.
\end{equation*}
The remaining case when $j$ is even and $i = j-1$ can be dealt with in a similar fashion, applying property 3) iii).
\end{proof}

\section{Improved parameter towers}

The properties detailed in Lemma \ref{dendrinosstovall}, though useful, are insufficient for the present purpose. Observe that although the even fibres of the tower constructed in Lemma \ref{dendrinosstovall} are two-dimensional sets, consisting of points $(r_{j/2}, t_j) \in [1,2] \times I$, all the bounds are decidedly one-dimensional in the sense that they are in terms of the $t_j$ variables and there is little reference to the dilation parameters. An additional refinement is necessary to take advantage the higher dimensionality of the even fibres.

\begin{lem}\label{towerlem} Fix $0 < \delta \ll 1$ a small parameter. There exists a point $(x_0, r_0) \in F$ and a tower $\{\Omega_k\}_{k = 1}^{d+1}$ satisfying all the properties of Lemma \ref{dendrinosstovall} with the additional property that for each even $1 < j \leq d+1$ either
\begin{equation*}
|t_j - t_{j-1}| \geq \delta (\alpha\beta)^{1/2} t_{j-1}^{-2K/d(d+1)}
\end{equation*}
holds for all $(\mathbf{r}_j, \mathbf{t}_j) \in \Omega_j$, or
\begin{equation*}
|t_j - t_{j-1}| < \delta (\alpha\beta)^{1/2} t_{j-1}^{-2K/d(d+1)} \quad \textrm{ and } \quad |r_{j/2} - r_{j/2-1}| \gtrsim (\beta/\alpha)^{1/2}
\end{equation*}
both hold for all $(\mathbf{r}_j, \mathbf{t}_j) \in \Omega_j$. If the former case the index $j$ is designated ``pre-red'', whilst in the latter $j$ is designated ``pre-blue''. The odd vertices are ``pre-achromatic''; that is, they are not assigned a ``pre-colour''. 
\end{lem}

\begin{rmk} The partition of the indices into the sets of odd, pre-red and pre-blue indices plays a very similar r\^ole to the construction of the band structure in \cite{Christ1998} and in particular the ``slicing method'' of \cite{Christ, Christ1998} will be utilised. 
\end{rmk}

\begin{rmk} As the prefix ``pre-'' suggests, later in the argument it will be convenient to relabel the indices. In particular, in the following section an updated labelling will be introduced which designates the indices either ``red'', ``blue'' or ``achromatic''. 
\end{rmk}

The result is essentially established as follows. By applying the argument of Dendrinos and Stovall \cite{Dendrinos} one obtains an initial tower with the properties stated in Lemma \ref{dendrinosstovall}. To ensure the additional property described in Lemma \ref{towerlem} holds one further refines the tower, appealing to the following elementary (but notationally-involved) result.

\begin{lem}\label{refinecor} Let $\{ \widetilde{\Omega}_j\}_{j=1}^D$ be a tower of height $D$ and for each even $1 < k \leq D$ let
\begin{equation*}
\big\{A_{k}(\mathbf{r}_{k-1}, \mathbf{t}_{k-1}) : (\mathbf{r}_{k-1}, \mathbf{t}_{k-1}) \in \widetilde{\Omega}_{k-1} \big\}
\end{equation*}
 a collection of measurable subsets of $[1,2] \times I$. Then there exists a tower\footnote{Strictly speaking, $\{\Omega_j\}_{j=1}^D$ will only satisfy a weak definition of a tower: see Remark \ref{measure remark} below.} $\{\Omega_j\}_{j=1}^D$ satisfying:
\begin{enumerate}[a)]
\item $\Omega_1 \subseteq \widetilde{\Omega}_1$ and $\Omega_j(\mathbf{r}_{j-1}, \mathbf{t}_{j-1}) \subseteq \widetilde{\Omega}_j(\mathbf{r}_{j-1}, \mathbf{t}_{j-1})$ for all $(\mathbf{r}_{j-1}, \mathbf{t}_{j-1}) \in \Omega_{j-1}$ and all $1 < j \leq D$;
\item The following estimates hold for the fibres:
\begin{enumerate}[i)]
\item $\mu_P(\Omega_1) \geq 2^{-\floor{D/2}} \mu_P(\widetilde{\Omega}_1)$.
\item Whenever $1 < j \leq D$ is odd,
\begin{equation*}
\mu_P\big(\Omega_j(\mathbf{r}_{j-1}, \mathbf{t}_{j-1})\big) \geq 2^{-\floor{D/2}} \mu_P\big(\widetilde{\Omega}_j(\mathbf{r}_{j-1}, \mathbf{t}_{j-1})\big)
\end{equation*}
holds for all $(\mathbf{r}_{j-1}, \mathbf{t}_{j-1}) \in \Omega_{j-1}$.
\item Whenever $1 < j \leq D$ is even, 
\begin{equation*}
\nu_P\big(\Omega_j(\mathbf{r}_{j-1}, \mathbf{t}_{j-1})\big) \geq 2^{-\floor{D/2}} \nu_P\big(\widetilde{\Omega}_j(\mathbf{r}_{j-1}, \mathbf{t}_{j-1})\big)
\end{equation*}
holds for all $(\mathbf{r}_{j-1}, \mathbf{t}_{j-1}) \in \Omega_{j-1}$.
\end{enumerate}
\item For each even $1 < k \leq D$ precisely one of the following holds:
\begin{enumerate}[i)] 
\item $\Omega_{k}(\mathbf{r}_{k-1}, \mathbf{t}_{k-1}) \subseteq A_k(\mathbf{r}_{k-1}, \mathbf{t}_{k-1})$ for all $(\mathbf{r}_{k-1}, \mathbf{t}_{k-1}) \in \Omega_{k-1}$;
\item $\Omega_{k}(\mathbf{r}_{k-1}, \mathbf{t}_{k-1}) \cap A_k(\mathbf{r}_{k-1}, \mathbf{t}_{k-1}) = \emptyset$ for all $(\mathbf{r}_{k-1}, \mathbf{t}_{k-1}) \in \Omega_{k-1}$.
\end{enumerate}
\end{enumerate}
\end{lem}

\begin{rmk}\label{measure remark} Strictly speaking, the proof below will not address the issue of whether the $\Omega_j$ are measurable (as required in the definition of a tower), although the fibres certainly will be (and therefore b) i) and b) ii) make sense). In practice, when the lemma is applied below it will be clear from the choice of sets $A_{k}(\mathbf{r}_{k-1}, \mathbf{t}_{k-1})$ that the resulting $\Omega_j$ are measurable and so this omission is unimportant for the present purpose.
\end{rmk}

\begin{proof}[Proof (of Lemma \ref{refinecor})] Proceed by induction on $D$, the case $D = 1$ being vacuous. Let $1 < D$ and fix a tower $\{\widetilde{\Omega}_j\}_{j=1}^D$. Apply the induction hypothesis to $\{\widetilde{\Omega}_j\}_{j=1}^{D-1}$ to obtain a tower $\{\widehat{\Omega}_j\}_{j=1}^{D-1}$ satisfying the properties a), b) and c) of the  corollary, with $D$ replaced by $D-1$. For each $(\mathbf{r}_{D-1}, \mathbf{t}_{D-1}) \in \widehat{\Omega}_{D-1}$ define $\widehat{\Omega}_{D}(\mathbf{r}_{D-1}, \mathbf{t}_{D-1}) := \widetilde{\Omega}_{D}(\mathbf{r}_{D-1}, \mathbf{t}_{D-1}) $. If $D$ is odd define $\widehat{\Omega}_{D}$ to be
\begin{equation*}
\big\{ (\mathbf{r}_{D}, \mathbf{t}_{D}) \in \widetilde{\Omega}_D : t_D \in \omega_{D}(\mathbf{r}_{D-1}, \mathbf{t}_{D-1})\textrm{ and } (\mathbf{r}_{D-1}, \mathbf{t}_{D-1}) \in \widehat{\Omega}_{D-1} \big\}
\end{equation*}
where $\mathbf{r}_{D} = \mathbf{r}_{D-1}$ and $\mathbf{t}_{D} = (\mathbf{t}_{D-1}, t_{D})$; throughout this article, similar notation will be used for elements belonging to levels of various parameter towers without further comment. Similarly, if $D$ is even, then define $\widehat{\Omega}_{D}$ to be
\begin{equation*}
\big\{ (\mathbf{r}_{D}, \mathbf{t}_{D}) \in \widetilde{\Omega}_D : (r_{D/2}, t_D) \in \omega_{D}(\mathbf{r}_{D-1}, \mathbf{t}_{D-1})\textrm{ and } (\mathbf{r}_{D-1}, \mathbf{t}_{D-1}) \in \widehat{\Omega}_{D-1} \big\}.
\end{equation*}

If $D$ is odd, then the proof is immediately completed by letting $\Omega_j := \widehat{\Omega}_j$ for $j=1,\dots, D$. It remains to consider the case when $D$ is even, which is more involved. Define a sequence of sets $\omega_{D-k} \subseteq \widehat{\Omega}_{D-k}$ for $1 \leq k \leq D-1$ recursively as follows. For all $(\mathbf{r}_{D-1}, \mathbf{t}_{D-1})  \in \widehat{\Omega}_{D-1}$ let
\begin{equation*}
\omega_D(\mathbf{r}_{D-1}, \mathbf{t}_{D-1}) := \widehat{\Omega}_D(\mathbf{r}_{D-1}, \mathbf{t}_{D-1}) \cap A(\mathbf{r}_{D-1}, \mathbf{t}_{D-1})
\end{equation*}
and define $\omega_{D-1}$ to be the set
\begin{equation*}
\left\{ (\mathbf{r}_{D-1}, \mathbf{t}_{D-1}) \in \widehat{\Omega}_{D-1} : \nu_P\big(\omega_D(\mathbf{r}_{D-1}, \mathbf{t}_{D-1})\big) \geq \frac{1}{2}\nu_P\big(\widehat{\Omega}_D(\mathbf{r}_{D-1}, \mathbf{t}_{D-1})\big) \right\}.
\end{equation*}
Hence, $\omega_{D-1}$ is the set of points $(\mathbf{r}, \mathbf{t}) \in \widehat{\Omega}_{D-1}$ with the property that most of the associated fibre $\widehat{\Omega}_D(\mathbf{r}, \mathbf{t})$ lies in $A(\mathbf{r}, \mathbf{t})$.

Now suppose $\omega_{D-k}$ has been defined for some $1 \leq k \leq D-2$ and let $(\mathbf{r}_{D-k-1}, \mathbf{t}_{D-k-1})  \in \widehat{\Omega}_{D-k-1}$. If $k$ is odd, then $D- k$ is also odd and $\omega_{D-k}(\mathbf{r}_{D-k-1}, \mathbf{t}_{D-k-1})$ is defined to be
\begin{equation*}
\big\{ t_{D-k} \in \widehat{\Omega}_{D-k}(\mathbf{r}_{D-k-1}, \mathbf{t}_{D-k-1}) : (\mathbf{r}_{D-k}, \mathbf{t}_{D-k}) \in \omega_{D-k}\big\}.
\end{equation*}
Let
\begin{equation*}
\omega_{D-k-1} := \left\{ (\mathbf{r}, \mathbf{t}) \in \widehat{\Omega}_{D-k-1} : \mu_P\big(\omega_{D-k}(\mathbf{r}, \mathbf{t})\big) \geq \frac{1}{2}\mu_P\big(\widehat{\Omega}_{D-k}(\mathbf{r}, \mathbf{t})\big) \right\}
\end{equation*}
so that $\omega_{D - k -1}$ is the set of points $(\mathbf{r}, \mathbf{t}) \in \widehat{\Omega}_{D-k -1}$ with the property that most of the associated fibre $\widehat{\Omega}_{D-k}(\mathbf{r}, \mathbf{t})$ lies in $\omega_{D-k}(\mathbf{r}, \mathbf{t})$.

Similarly, if $k$ is even, then $D - k$ is even and $\omega_{D-k}(\mathbf{r}_{D-k-1}, \mathbf{t}_{D-k-1})$ is defined to be
\begin{equation*}
\big\{ (r_{(D-k)/2}, t_{D-k}) \in \widehat{\Omega}_{D-k}(\mathbf{r}_{D-k-1}, \mathbf{t}_{D-k-1}) : (\mathbf{r}_{D-k}, \mathbf{t}_{D-k}) \in \omega_{D-k}\big\}
\end{equation*}
and one completes the recursive definition by letting
\begin{equation*}
\omega_{D-k-1} := \left\{ (\mathbf{r}, \mathbf{t}) \in \widehat{\Omega}_{D-k-1} : \nu_P\big(\omega_{D-k}(\mathbf{r}, \mathbf{t})\big) \geq \frac{1}{2}\nu_P\big(\widehat{\Omega}_{D-k}(\mathbf{r}, \mathbf{t})\big) \right\}.
\end{equation*}
Having constructed the sequence $\omega_{D-k}$ for $1 \leq k \leq D-1$, suppose $\mu_P(\omega_1) \geq \tfrac{1}{2}\mu_P(\widehat{\Omega}_1)$. If one defines $\Omega_1 := \omega_1$ and $\Omega_j(\mathbf{r}_{j-1}, \mathbf{t}_{j-1}) := \omega_j(\mathbf{r}_{j-1}, \mathbf{t}_{j-1})$ for $1 < j \leq D$, then one may construct a tower inductively by setting
\begin{equation}\label{tower recursion 1}
\Omega_j := \{(\mathbf{r}_{j}, \mathbf{t}_{j}) \in [1,2]^{(j-1)/2} \times [0,1]^{j} : (\mathbf{r}_{j-1}, \mathbf{t}_{j-1}) \in \Omega_{j-1} \textrm{ and } t_j \in \Omega_j(\mathbf{r}_{j-1}, \mathbf{t}_{j-1}) \}
\end{equation}
for $j > 1$ odd and
\begin{equation}\label{tower recursion 2}
\Omega_j := \{(\mathbf{r}_{j}, \mathbf{t}_{j}) \in [1,2]^{j/2} \times [0,1]^{j} : (\mathbf{r}_{j-1}, \mathbf{t}_{j-1}) \in \Omega_{j-1} \textrm{ and } (r_{j/2},t_j) \in \Omega_j(\mathbf{r}_{j-1}, \mathbf{t}_{j-1}) \}
\end{equation}
for $j$ even. It immediately follows from the definitions that the resulting tower $\{\Omega_j\}_{j-1}^{D}$ satisfies the properties stated in the lemma with c) i) holding for $k=D$. 

On the other hand, if $\mu_P(\omega_1) < \tfrac{1}{2}\mu_P(\widehat{\Omega}_1)$, then define $\Omega_1 := \widehat{\Omega}_1 \setminus \omega_1$ and let
\begin{equation*}
\Omega_j(\mathbf{r}_{j-1}, \mathbf{t}_{j-1}) := \widehat{\Omega}_j(\mathbf{r}_{j-1}, \mathbf{t}_{j-1}) \setminus\omega_j(\mathbf{r}_{j-1}, \mathbf{t}_{j-1})
\end{equation*}
for $1 < j \leq D$ so that properties a) and b) i) and c) clearly hold for the resulting tower $\{\Omega_j\}_{j-1}^{D}$, which is again defined by \eqref{tower recursion 1} and \eqref{tower recursion 2}. To prove b) ii), suppose $1 < j \leq D$ is odd and $(\mathbf{r}_{j-1}, \mathbf{t}_{j-1}) \in \Omega_{j-1}$. Thus, $(r_{(j-1)/2}, t_{j-1}) \in \Omega_{j-1}(\mathbf{r}_{j-2}, \mathbf{t}_{j-2})$ and, by the definition of $\omega_{j-1}(\mathbf{r}_{j-2}, \mathbf{t}_{j-2})$, it follows that $(\mathbf{r}_{j-1}, \mathbf{t}_{j-1}) \in \widehat{\Omega}_{j-1} \setminus\omega_{j-1}$. Finally, the definition of $\omega_{j-1}$ ensures
\begin{align*}
\mu_P\big(\Omega_{j}(\mathbf{r}_{j-1}, \mathbf{t}_{j-1})\big) &=\mu_P\big(\widehat{\Omega}_{j}(\mathbf{r}_{j-1}, \mathbf{t}_{j-1})\big) - \mu_P\big(\omega_{j}(\mathbf{r}_{j-1}, \mathbf{t}_{j-1})\big) \\
&\geq 2^{-1}\mu_P\big(\widehat{\Omega}_{j}(\mathbf{r}_{j-1}, \mathbf{t}_{j-1})\big)\\
&\geq 2^{-\floor{(D-1)/2}+1}\mu_P\big(\widetilde{\Omega}_{j}(\mathbf{r}_{j-1}, \mathbf{t}_{j-1})\big),
\end{align*}
where the induction hypothesis has been applied in the last inequality. A similar argument shows b) iii) also holds, completing the proof. 

\end{proof}

Having stated this refinement result one may proceed to prove Lemma \ref{towerlem}. 

\begin{proof}[Proof (of Lemma \ref{towerlem})] Observe, defining $\Sigma := \R^d \times [1,2] \times I$ it follows that
\begin{equation*}
\langle A \chi_E, \chi_F \rangle = \int_{\Sigma} \chi_{U}(x,r,t)\,\lambda_P(t) \ud x \ud r\ud t
\end{equation*}
where $\lambda_P(t):= t^{2K/d(d+1)}$ and
\begin{equation*}
U := \big\{ (x,r,t) \in \Sigma : x - r P(t) \in E \textrm{ and } (x,r) \in F \big\}.
\end{equation*}
Writing $I := (a,b)$ where $a \geq 0$ and $b - a =1$, a method due to Dendrinos and Stovall \cite{Dendrinos} may be applied to produce a sequence $\{U_k\}_{k=0}^{\infty}$ of decreasing subsets of $U$ of pairwise comparable measure such that for all $k \geq 1$ either
\begin{equation}\label{ds1}
\int_t^b \chi_{U_{k-1}}(x , r, \tau)\, \lambda_P(\tau)\ud \tau \geq 4^{-(k+1)} \alpha
\end{equation}
or
\begin{equation}\label{ds2}
\int_1^2 \int_t^b \chi_{U_{k-1}}(x - r P(t) +\rho P(\tau), \rho, \tau)\, \lambda_P(\tau)\ud \tau \ud \rho \geq 4^{-(k+1)} \beta
\end{equation}
holds for all $(x,r,t) \in U_k$. Specifically, the $\{U_k\}_{k=0}^{\infty}$ can be chosen so that
\begin{equation*}
\int_{\Sigma} \chi_{U_0}(x,r,t)\,\lambda_P(t) \ud x \ud r\ud t \geq \frac{1}{2}\langle A \chi_E, \chi_F \rangle
\end{equation*}
and for all $k \geq 1$:
\begin{enumerate}[i)]
\item $U_k \subseteq U_{k-1}$ and
\begin{equation*}
\int_{\Sigma} \chi_{U_{k}}(x,r,t) \,\lambda_P(t) \ud x\ud r\ud t  \geq \frac{1}{4}\int_{\Sigma} \chi_{U_{k-1}}(x,r,t) \,\lambda_P(t)\ud x\ud r\ud t;
\end{equation*}
\item If $k \not\equiv d \mod 2$ (respectively, $k \equiv d \mod 2$), then \eqref{ds1} (respectively \eqref{ds2}) holds for all $(x,r,t) \in U_k$.
\item Furthermore, for each $k$, if $t \in I$ is such that $(x,r,t) \in U_k$ for some $(x,r) \in \R^d \times [1,2]$, then $t \geq  (\alpha/2 \kappa)^{\kappa}$ where $\kappa$ is as in \eqref{kappa}.  
\end{enumerate}

This construction is due to Dendrinos and Stovall \cite{Dendrinos}, however the details are appended for completeness.

Fix $(x_0, r_0, t_0) \in U_{d+1}$. The next step is to use the sets $\{U_k\}_{k=0}^{d+1}$ to construct an initial tower $\{\widetilde{\Omega}_j\}_{j=1}^{d+1}$ satisfying the properties of Lemma \ref{dendrinosstovall} and such that whenever $(\mathbf{r}_j, \mathbf{t}_j) \in \widetilde{\Omega}_j$ for some $1 \leq j \leq d+1$, it follows that
\begin{equation*}
\left\{\begin{array}{ll}
\left(\Psi_j(\mathbf{r}_j, \mathbf{t}_j), r_{j/2}, t_{j}\right) \in U_{d+1-j} & \textrm{if $0 \leq j \leq d+1$ is even}\\
&\\
\left(\Psi_{j-1}(\mathbf{r}_{j-1}, \mathbf{t}_{j-1}), r_{\floor{j/2}}, t_{j}\right) \in U_{d+1-j} & \textrm{if $1 \leq j \leq d+1$ is odd}
\end{array}\right.
\end{equation*}
where $\Psi_j(\mathbf{r}_j, \mathbf{t}_j) := \Psi_j(x_0, r_0;\mathbf{r}_j, \mathbf{t}_j)$ for $j \geq 1$ is as defined in \eqref{parf} and $\Psi_0(\mathbf{r}_0, \mathbf{t}_0) := x_0$. It is convenient to consider $(\mathbf{r}_0, \mathbf{t}_0)$ as some arbitrary object (say, $(\mathbf{r}_0, \mathbf{t}_0) := (r_0, t_0)$) and $\Psi_0$ as a function on the singleton set $\widetilde{\Omega}_0 := \{(\mathbf{r}_0, \mathbf{t}_0)\}$, taking the value $x_0$.

The tower $\{\widetilde{\Omega}_j\}_{j=1}^{d+1}$ is constructed recursively. To begin, define $\widetilde{\Omega}_0$ as above; suppose $\widetilde{\Omega}_{j}$ has been defined for some $0 \leq j \leq d$ and fix $(\mathbf{r}_{j}, \mathbf{t}_{j}) \in \widetilde{\Omega}_{j}$. The argument splits into two cases, depending on the parity of $j$.

\subsection*{Case 1: $j$ is even.} Since $(\Psi_{j}(\mathbf{r}_{j}, \mathbf{t}_{j}), r_{j/2}, t_{j}) \in U_{d+1-j}$ and $d+1-j \not\equiv d \mod 2$, one may apply \eqref{ds1} to deduce
\begin{equation}\label{ds3}
 \int_{t_{j}}^b \chi_{U_{d-j}}(\Psi_{j}(\mathbf{r}_{j}, \mathbf{t}_{j}), r_{j/2}, \tau)\, \lambda_P(\tau)\ud \tau \geq 4^{-(d+2-j)} \alpha.
\end{equation}
It is therefore possible to choose $t_j < s_j < b$ with the property
\begin{equation}\label{ds4}
\int_{t_j}^{s_j} \,\lambda_P(\tau)\ud \tau = 4^{-(d+5/2-j)} \alpha.
\end{equation}
Define the set
\begin{equation*}
\widetilde{\Omega}_{j+1}(\mathbf{r}_j, \mathbf{t}_j) := \big\{ t_{j+1} \in (s_j, b] : (\Psi_{j}(\mathbf{r}_{j}, \mathbf{t}_{j}), r_{j/2}, t_{j+1}) \in U_{d-j} \big\},
\end{equation*}
noting, by \eqref{ds3} and \eqref{ds4}, that this has measure $\mu_P(\widetilde{\Omega}_{j+1}(\mathbf{r}_j, \mathbf{t}_j)) \geq 4^{-(d+5/2-j)} \alpha$. To complete the recursive step in this case, if $j=0$ let $\widetilde{\Omega}_1 := \widetilde{\Omega}_1(\mathbf{r}_0, \mathbf{t}_0)$ whilst for $j > 0$ define
\begin{equation*}
\widetilde{\Omega}_{j+1} := \big\{ (\mathbf{r}_{j+1}, \mathbf{t}_{j+1} ) : (\mathbf{r}_{j}, \mathbf{t}_{j} ) \in \widetilde{\Omega}_j \textrm{ and } t_{j+1} \in \widetilde{\Omega}_{j+1}(\mathbf{r}_j, \mathbf{t}_j) \big\}.
\end{equation*}

\subsection*{Case 2: $j$ is odd.} Since $(\Psi_{j-1}(\mathbf{r}_{j-1}, \mathbf{t}_{j-1}), r_{\floor{j/2}}, t_{j}) \in U_{d+1-j}$ and $d+1-j \equiv d \mod 2$, one may apply \eqref{ds2} to deduce
\begin{equation}\label{ds5}
\int_1^2 \int_{t_{j}}^b \chi_{U_{d-j}}(\Psi_{j}(\mathbf{r}_{j}, \mathbf{t}_{j}) +\rho P(\tau), \rho, \tau)\, \lambda_P(\tau)\ud \tau \ud \rho \geq 4^{-(d+2-j)} \beta.
\end{equation}
Here the identity 
\begin{equation*}
\Psi_{j}(\mathbf{r}_{j}, \mathbf{t}_{j}) = \Psi_{j-1}(\mathbf{r}_{j-1}, \mathbf{t}_{j-1}) - r_{\floor{j/2}}P(t_j)
\end{equation*}
has been applied, which is a consequence of the definition \eqref{parf}. It is therefore possible to choose $t_j < s_j < b$ and $t_j < \tilde{s}_j$ with the properties
\begin{equation}\label{ds6}
\int_{t_j}^{s_j} \,\lambda_P(\tau)\ud \tau = 4^{-(d+5/2-j)} \beta, \quad \int_{t_j}^{\tilde{s}_j} \,\lambda_P(\tau)\ud \tau = 2^{-(d+3-j)} (\alpha\beta)^{1/2}.
\end{equation}
Define $\widetilde{\Omega}_{j+1}(\mathbf{r}_j, \mathbf{t}_j)$ to be the set
\begin{equation*}
\left\{ (r_{\ceil{j/2}}, t_{j+1}) \in D(\mathbf{r}_j, \mathbf{t}_j) : \left(\Psi_{j}(\mathbf{r}_{j}, \mathbf{t}_{j}) + r_{\ceil{j/2}} P(t_{j+1}), r_{\ceil{j/2}}, t_{j+1}\right) \in U_{d-j} \right\}
\end{equation*}
where $D(\mathbf{r}_j, \mathbf{t}_j) := ([1,2] \times (s_j, b])\setminus R(\mathbf{r}_j, \mathbf{t}_j)$ for the rectangle
\begin{equation*}
R(\mathbf{r}_j, \mathbf{t}_j):=\big\{ (r,t) \in \R^2 : |r - r_{\floor{j/2}}| \leq 2^{-(d+4-j)} (\beta/\alpha)^{1/2} \textrm{ and } t_j \leq t \leq \tilde{s}_j \big\}.
\end{equation*}
Observe, by \eqref{ds6} it follows that
\begin{equation*}
\int_1^2 \int_I \chi_{R(\mathbf{r}_j, \mathbf{t}_j)}(r,t)\,\lambda_P(t)\ud t \ud r \leq 4^{-(d+3-j)}\beta.
\end{equation*}
Thus, by \eqref{ds5} and \eqref{ds6}, the set $\widetilde{\Omega}_{j+1}(\mathbf{r}_j, \mathbf{t}_j)$ has measure $\nu_P(\widetilde{\Omega}_{j+1}(\mathbf{r}_j, \mathbf{t}_j))\geq 4^{-(d+3-j)} \beta$.

Finally, to complete the recursive definition let
\begin{equation*}
\widetilde{\Omega}_{j+1} := \big\{ (\mathbf{r}_{j+1}, \mathbf{t}_{j+1} ) : (\mathbf{r}_{j}, \mathbf{t}_{j} ) \in \Omega_j \textrm{ and } (r_{\ceil{j/2}}, t_{j+1}) \in \Omega_{j+1}(\mathbf{r}_j, \mathbf{t}_j) \big\}.
\end{equation*}

It is easy to verify the collection $\{\widetilde{\Omega}_j\}_{j=1}^{d+1}$ forms a tower satisfying all the properties of Lemma \ref{dendrinosstovall}. Finally, Lemma \ref{refinecor} is applied to further refine this tower to ensure the additional property stated in Lemma \ref{towerlem} holds. For each $1  < j \leq d+1$ even, define
\begin{equation*}
A(\mathbf{r}_{j-1}, \mathbf{t}_{j-1}) := \left\{ (r, t) \in [1,2] \times I : t - t_{j-1} > \delta (\alpha\beta)^{1/2}t_{j-1}^{-2K/d(d+1)} \right\}
\end{equation*} 
for all $(\mathbf{r}_{j-1}, \mathbf{t}_{j-1}) \in \widetilde{\Omega}_{j-1}$. Letting $\{\Omega_j\}_{j=1}^{d+1}$ denote the refined tower, the existence of which is guaranteed by Lemma \ref{refinecor}, it is easy to see this has all the desired properties. In particular, for each $1 < j \leq d+1$ even, precisely one of the following holds:
\begin{enumerate}[a)] 
\item For all $(\mathbf{r}_j, \mathbf{t}_j) \in \Omega_{j}$ one has $|t_j - t_{j-1}| > \delta (\alpha\beta)^{1/2}t_{j-1}^{-2K/d(d+1)}$.
\item For all $(\mathbf{r}_j, \mathbf{t}_j) \in \Omega_{j}$ one has $|t_j - t_{j-1}| \leq \delta (\alpha\beta)^{1/2}t_{j-1}^{-2K/d(d+1)}$. In this case, observe for $t_{j-1} \leq \tau \leq t_j$ it follows that
\begin{equation*}
|\tau - t_{j-1}| \leq \delta (\alpha\beta)^{1/2}t_{j-1}^{-2K/d(d+1)} \lesssim \delta \alpha (\alpha^{\kappa})^{-2K/d(d+1)} \lesssim \delta t_{j-1}
\end{equation*}
by condition iii) of the sets $U_k$ described above. Thus $\tau \lesssim t_{j-1}$ and consequently
\begin{equation*}
\int_{t_{j-1}}^{t_j}\,\lambda_P(\tau)\ud \tau \lesssim \delta (\alpha\beta)^{1/2}.
\end{equation*}
If $\delta$ is chosen from the outset to be sufficiently small depending only on $d$ and $K$, then it follows that $t_{j-1} < t_j < \tilde{s}_{j-1}$ by the preceding inequality and the definition of $\tilde{s}_{j-1}$ from \eqref{ds6}. Since $(r_{j/2}, t_j) \notin R(\mathbf{r}_{j-1}, \mathbf{t}_{j-1})$, one concludes that $|r_{j/2} - r_{j/2-1}| > 2^{-(d+5-j)} (\beta / \alpha)^{1/2}$.
\end{enumerate}
\end{proof}

\section{Definition of the parameter domain}\label{definitionparameter}

Henceforth fix a tower $\{\Omega_j\}_{j=1}^{d+1}$ satisfying the properties stated in Lemma \ref{towerlem} with a suitably small choice of $0 < \delta \ll 1$ so as to satisfy all the forthcoming requirements of the proof. It is stressed that the subsequent argument will require $\delta$ to be chosen depending only on the admissible parameters $d$ and $\deg P$; a careful examination of what follows shows such a choice of $\delta$ is always possible. 

Observe that the set $\Omega_{d+1}$ is of dimension $\floor{3(d+1)/2}$; one requires a domain of either dimension $d$ or $d+1$ to effectively parametrise either the set $E$ or $F$. Two methods can be applied to remedy this excess of variables. The first is to simply consider the tower defined only to a lower level; that is, only work with $\{\Omega_j\}_{j=1}^N$ for some $N \leq d+1$. The second method is to consider the whole tower $\{\Omega_j\}_{j=1}^{d+1}$ and \emph{freeze} a number of the variables $t_j$. What is essentially meant by this is that for some set of indices $\mathcal{I} \subset \{1, \dots, d+1\}$ and choice of $(s_i)_{i\in \mathcal{I}} \subset \R^{\# \mathcal{I}}$ each set $\Omega_j$ is replaced with 
\begin{equation*}
\big\{ (r, t) \in \Omega_j : t_i = s_i \textrm{ for all }i \in \mathcal{I} \cap \{1, \dots, j\} \big\}.
\end{equation*}
In order to optimise the subsequent Jacobian estimates, both methods are combined below and the variables to be frozen are chosen according to the ``pre-colour'' of their indices. In particular only $t_i$ for $i$ a pre-blue index will be frozen. In light of this discussion define the function $\zeta: \{1, \dots, d+1\} \rightarrow \{1,2\}$ as follows:
\begin{equation*}
\zeta(j) := \left\{\begin{array}{ll} 
2 & \textrm{if $j$ is pre-red} \\
1 & \textrm{otherwise}
\end{array} \right. \qquad \textrm{for $j = 1, \dots, d+1$.}
\end{equation*}
One can think of $\zeta(j)$ as the number of variables contributed by the fibres of the $j$th floor of the tower after the pre-blue variables have been frozen. Note that there exists a least $1 < N \leq d+1$ such that 
\begin{equation}\label{defn N}
Z(N) := \sum_{j=1}^{N} \zeta(j) \in \{d, d+1\};
\end{equation}
in particular, the number of variables contributed by the first $N$ floors corresponds to the dimension of either $E$ or $F$. At this point the proof splits into a number of different cases depending on the parity of $N$ and the value of $Z(N)$.

\begin{dfn} If $N$ is odd, then $\mathrm{Case}(1, Z(N))$ holds. If $N$ is even, then $\mathrm{Case}(2, Z(N))$ holds. 
\end{dfn}

Ostensibly there are four distinct cases; however, the minimality of $N$ precludes $\mathrm{Case}(1, d+1)$ and so it suffices to consider the remaining three cases.

The preceding observations lead to the definition of a suitable parameter domain $\Omega$ and mapping $\Phi$ which will form the focus of study for the remainder of the paper. By the nature of the definitions of $\Omega$ and $\Phi$, it will be convenient to introduce a relabelling of the relevant indices as ``red,'' ``blue'' or ``achromatic'' to replace the established ``pre-red, pre-blue, pre-achromatic'' system. All these definitions depend on which $\mathrm{Case}(i, j)$ happens to hold.

\subsection*{Case$\mathbf{(1, d)}$ and Case$\mathbf{(2, d+1)}$:} The simplest situation corresponds to when either $\mathrm{Case}(1, d)$ or $\mathrm{Case}(2, d+1)$ holds. In both instances one defines $\Omega := \Omega_{N}$ and $\Phi := \Phi_{N}$. In the former case $\Phi : \Omega \rightarrow E$ whilst in the latter $\Phi : \Omega \rightarrow F$ by the properties 2. i) and 3. i) of Lemma \ref{dendrinosstovall}, respectively. Here essentially no relabelling is required: each pre-red (respectively, pre-blue) index $1 \leq j \leq N$ is designated red (respectively, blue) whilst the odd indices are achromatic (that is, they have no colour assigned to them). 

\subsection*{Case$\mathbf{(2, d)}$:} This situation is slightly more complicated. Fix $t_0 \in \Omega_1$ and consider the family of sets $\{\Omega_j^*\}_{j=1}^{N}$ defined by
\begin{equation*}
\Omega_j^* := \big\{ (t_1, \dots, t_{j}) : (t_0, t_1, \dots, t_{j}) \in \Omega_{j+1} \big\} \qquad \textrm{for $j = 1, \dots, N$.}
\end{equation*}
It is easy to see $\{\Omega_j^*\}_{j=1}^{N}$ constitutes a type 2 tower. Let $y_0 := x_0 - r_0P(t_0)$ and let $\{\Phi^*_j\}_{j=1}^N$ denote the family of mappings associated to $\{\Omega_j^*\}_{j=1}^{N}$ and the point $y_0$, as defined in Definition \ref{associated mappings}. 
Define $\Omega := \Omega^*_N$ and $\Phi := \Phi^*_{N}$ and observe, by property 2. i) of Lemma \ref{dendrinosstovall}, that $\Phi : \Omega \rightarrow E$. In this case the colouring system of the indices is redefined. In particular,
\begin{enumerate}[i)]
\item The index 1 is designated red;
\item The odd indices $1 < j \leq N$ are designated red (respectively, blue) if $j+1$ was pre-red (respectively, pre-blue) in the previous scheme;
\item The even indices are designated achromatic. 
\end{enumerate}

\section{Freezing variables and families of parametrisations}\label{freeze}

Rather than use a single function to parametrise $E$ or $F$, here the slicing method of Christ \cite{Christ, Christ1998} is used to construct a family of maps $G_{\sigma}$. To begin some notation is introduced.

Let $M$ denote the number of non-blue indices in $\{1, \dots, N\}$ and label these indices $l_1 < l_2 < \dots < l_M$. Similarly, let $m$ (respectively, $n$) denote the number of red (respectively, blue) indices so that $M = N - n$.  For the reader's convenience the following table indicates the relationship between these parameters in the various cases.

\begin{equation}\label{counting variables}
    \begin{tabular}{ | c | c | c |}
    \hline
    & N & d  \\ \hline
    $\mathrm{Case}(1,d)$ & 2m +2n + 1 & 3m + 2n + 1  \\ \hline
    $\mathrm{Case}(2,d+1)$ & 2m +2n & 3m + 2n - 1  \\ \hline
    $\mathrm{Case}(2,d)$ & 2m +2n & 3m + 2n \\
    \hline
    \end{tabular}
\end{equation}
\vspace{10pt}

These computations follow immediately from the definition of $N$ in \eqref{defn N}.

Now let $ 1 \leq \mu_1 < \dots < \mu_m \leq M$ be such that $\{ l_{\mu_i}\}_{i=1}^m$ is precisely the set of red indices. Rather than the blue indices themselves, it will be useful to enumerate those indices which lie directly before a blue index. Irrespective of which case happens to hold, any blue index is at least 2 and so there are precisely $n$ indices lying directly before a blue index. In particular, let $1 \leq \nu_1 < \dots < \nu_{n} \leq M$ be such that $\{l_{\nu_j}+1\}_{j=1}^{n}$ are precisely the blue indices.

Define functions $\tau$ and $\sigma$ on $\Omega$ by
\begin{eqnarray*}
\tau = (\tau_1, \dots, \tau_M) &:=& ( t_{l_1}, \dots, t_{l_M}) \\
s = (s_{\nu_1}, \dots, s_{\nu_{n}}) &:=& (t_{l_{\nu_j}+1} - t_{l_{\nu_j}})_{j=1}^{n} \\
\sigma = (\sigma_{\nu_1}, \dots, \sigma_{\nu_n}) &:=& (s_{\nu_j} \tau_{\nu_j}^{2K/d(d+1)})_{j=1}^{n}.
\end{eqnarray*}

Finally let $\rho = (\rho_{\mu_1}, \dots, \rho_{\mu_m}, \rho_{\nu_1}, \dots, \rho_{\nu_{n}})$ where $\rho_{\mu_i}$ (respectively, $\rho_{\nu_j}$) is the dilation variable arising from the fibres of floor $l_{\mu_i}$ (respectively, $l_{\nu_j} +1$) of the tower. More precisely, $\rho_{\mu_j} := r_{\ceil{l_{\mu_i}/2}}$ for $1 \leq i \leq m$ whilst $\rho_{\nu_j} := r_{\ceil{(l_{\nu_j}+1)/2}}$ for $1 \leq j \leq n$.

Observe that the map 
\begin{equation}\label{phi change of variables}
\varphi: (r, t) \mapsto (\rho, \tau, \sigma)
\end{equation}
is a valid change of variables with Jacobian determinant satisfying
\begin{equation*}
\left|\det \frac{\partial \varphi}{\partial (r, t)}\right| = \prod_{j=1}^{n} \tau_{\nu_j}^{2K/d(d+1)}.
\end{equation*}
For $\sigma \in \R^n$ define the parameter set $\omega(\sigma) := \big\{ (\rho, \tau) : \varphi^{-1}(\rho, \tau, \sigma) \in \Omega \big\}$ and  let
\begin{equation}\label{setw}
W := \big\{ \sigma \in \R^{n} : \omega(\sigma) \neq \emptyset \} \subseteq \big[0, \delta(\alpha\beta)^{1/2}\big]^n
\end{equation}
where the inclusion follows from properties of the blue indices. Finally, consider the mapping $G_{\sigma}$ on $\omega(\sigma)$ by
\begin{equation*}
G_{\sigma}(\rho, \tau):= \Phi \circ \varphi^{-1}(\rho, \tau, \sigma).
\end{equation*}
By \eqref{counting variables}, it follows that in $\mathrm{Case}(1, d)$ and $\mathrm{Case}(2, d)$  the maps $G_{\sigma}$ are functions of $d$ variables and take values in $E$ whilst in $\mathrm{Case}(2, d+1)$ the $G_{\sigma}$ are functions of $d+1$ variables and take values in $F$. Hence in each case the maps $G_{\sigma}$ have the desirable property that the domain and codomain are of equal dimension.

Furthermore, the polynomial nature of maps $G_{\sigma}$ imply each has bounded multiplicity. In particular, the following well-known multiplicity lemma applies to this situation.

\begin{lem}[Multiplicity Lemma]\label{multiplicity} Let $Q: \R^d \rightarrow \R^d$ be a polynomial mapping; that is, $Q(t) = (Q_j(t))_{j=1}^d$ for all $t \in \R^d$ where each $Q_j : \R^d \rightarrow \R$ is a polynomial in $d$ variables. Suppose the Jacobian determinant $J_Q$ of $Q$ does not vanish everywhere. Then for almost every $x \in \R^d$ the set $Q^{-1}(\{x\})$ is finite. Moreover, for almost every $x \in \R^d$ the inequality
\begin{equation}\label{mult1}
\# Q^{-1}(\{x\}) \leq \prod_{j=1}^d \deg (Q_j)
\end{equation}
holds, where $\deg(Q_j)$ denotes the degree of $Q_j$.
\end{lem}

The simple proof of this lemma appears in \cite{Christ1998}; however, it is included at the end of this section for completeness.

As a consequence of the Multiplicity Lemma, if $J_{\sigma}$ denotes the Jacobian of $G_{\sigma}$, then in $\mathrm{Case}(1, d)$ and $\mathrm{Case}(2, d)$ one concludes that the estimate
\begin{equation}\label{Jacobian estimate}
|E| \gtrsim \int_{\omega(\sigma)} |J_{\sigma}(\rho, \tau)| \,\ud \rho\ud \tau 
\end{equation}
holds for all $\sigma \in W$. In $\mathrm{Case}(2, d+1)$ there is a similar estimate but with $|E|$ replaced with $|F|$ on the left-hand side of the above expression. Thus, in order to establish either \eqref{equivweake} or \eqref{equivweakf} in the present cases it suffices to prove a suitable estimate for the Jacobian $|J_{\sigma}(\rho, \tau)|$ on the set $\omega(\sigma)$.

This section is concluded with the proof of the Multiplicity Lemma.

\begin{proof}[Proof (of the Multiplicity Lemma)] Since the zero locus $Z$ of $J_Q$ is a proper algebraic subset of $\R^d$ it has measure zero. Furthermore, as $Q$ is a polynomial (and therefore locally Lipschitz) mapping the image $Q(Z)$ of $Z$ under $Q$ has measure zero. It is claimed that \eqref{mult1} holds for all $x \in \R^d \setminus Q(Z)$. Indeed, fixing such an $x$ notice that $Q^{-1}(\{x\}) = \bigcap_{j=1}^d V_j^x$ where $\{V_j^x\}_{j=1}^d$ are algebraic sets given by
\begin{equation*}
V_j^x := \{ t \in \R^d : Q_j(t) - x_j = 0 \}.
\end{equation*}
Bezout's theorem implies that the cardinality of this intersection is either uncountable or at most $\prod_{j=1}^d \deg (Q_j)$.\footnote{Recall, Bezout's theorem states that for any collection $Q_1, \dots, Q_n$ of homogeneous polynomials on $\C\mathbb{P}^n$ the number of intersection points of the associated hypersurfaces $ \{z \in \C\mathbb{P}^n : Q_j(z) = 0 \}$ (counted with multiplicity) is either uncountable or precisely $\prod_{j=1}^n \deg (Q_j)$. The real version used here follows by homogenising the polynomials and taking the domain of the resulting functions to be $\mathbb{CP}^n$. One then applies Bezout's theorem in complex projective space, de-homogenises and restricts to real-value intersection points. See, for example, \cite{Hassett2007} pp 223-224.} It therefore suffices to show that $Q^{-1}(\{x\})$ is not uncountable; this is achieved by proving each point of the set is isolated. By the choice of $x$, whenever $t_0 \in Q^{-1}(\{x\})$ the vectors $\{ \nabla Q_j(t_0)\}_{j=1}^d$ span $\R^d$. Thus the $V_j^x$ are smooth hypersurfaces in a neighbourhood of $t_0$ which, of course, intersect at $t_0$ and are transversal at this 
point of intersection. It follows that $t_0$ must be an isolated point of $Q^{-1}(\{x\})$, as required. 
\end{proof}

\section{Reduction to Jacobian estimates}

Recall, in order to prove the main theorem it suffices to obtain a lower bound on either $|E|$ or $|F|$ in terms of $\alpha$ and $\beta$, as discussed in \eqref{equivweake} and \eqref{equivweakf}. In the previous sections,  it was shown that when either $\mathrm{Case}(1, d)$ or $\mathrm{Case}(2, d)$ holds a family of useful mappings $G_{\sigma} \colon \omega(\sigma) \to E$ can be constructed. These mappings effectively parametrise the set $E$ and, in particular, one has the estimate \eqref{Jacobian estimate}. A similar construction is available in $\mathrm{Case}(2, d+1)$ (this time parametrising the set $F$) and so it remains to find effective bounds for integrals such as that appearing in the right-hand side of \eqref{Jacobian estimate}. This will be achieved by estimating pointwise the Jacobian determinant $J_{\sigma}$ of the map $G_{\sigma}$. In order to state the main result in this direction, it is convenient to introduce the notation
\begin{equation*}
\eta := \left\{ \begin{array}{ll}
0 & \textrm{if either $\mathrm{Case}(1, d)$ or $\mathrm{Case}(2, d+1)$ holds} \\
1 & \textrm{if $\mathrm{Case}(2, d)$ holds}
\end{array}\right. .
\end{equation*}

\begin{lem}\label{jaclem1} Let $\sigma \in W$, where $W$ is as defined in \eqref{setw}. Then 
\begin{equation*}
|J_{\sigma}(\rho, \tau)| \gtrsim \alpha^{d(d+1)/2 - M}(\beta / \alpha )^{(m+n - \eta)/2} \prod_{l=1}^M \tau_l^{2K/d(d+1)}
\end{equation*}
for all $(\rho, \tau) \in \omega(\sigma)$. 
\end{lem}

Theorem \ref{bigthm} is a direct consequence of Lemma \ref{jaclem1}.

\begin{proof}[Proof (of Theorem \ref{bigthm}, assuming Lemma \ref{jaclem1})] To prove Theorem \ref{bigthm} it suffices to show the estimate \eqref{equivweake} holds in both $\mathrm{Case}(1, d)$ and $\mathrm{Case}(2,d)$ and \eqref{equivweakf} holds in $\mathrm{Case}(2, d+1)$. Indeed, recall both \eqref{equivweake} and \eqref{equivweakf} are equivalent to the desired endpoint restricted weak-type $(p_1, q_1)$ inequality \eqref{weak} for $A$. For notational convenience, let 
\begin{equation*}
|X| := \left\{ \begin{array}{ll}
|E| & \textrm{if either $\mathrm{Case}(1, d)$ or $\mathrm{Case}(2, d)$ holds} \\
|F| & \textrm{if $\mathrm{Case}(2, d+1)$ holds}
\end{array}\right. .
\end{equation*}

Apply Lemma \ref{jaclem1} to each $\sigma \in W$ together with the Multiplicity Lemma to deduce in all cases
\begin{eqnarray*}
|X|	&\gtrsim& \int_{\omega(\sigma)} |J_{\sigma}(\rho, \tau)| \,\ud \rho\ud \tau \\
&\gtrsim& \alpha^{d(d+1)/2 - M}(\beta / \alpha )^{(m+n - \eta)/2}\int_{\omega(\sigma)} \prod_{k=1}^M \tau_k^{2K/d(d+1)} \ud \rho\ud \tau .
\end{eqnarray*}
Integrating both sides of the preceding inequality over $W$, it follows that
\begin{equation}\label{wkpf2}
 (\alpha\beta)^{n/2} |X| \gtrsim \alpha^{d(d+1)/2 - M}(\beta / \alpha )^{(m+n - \eta)/2} \int_{\varphi(\Omega)} \prod_{k=1}^M \tau_k^{2K/d(d+1)} \,\ud \rho \ud \tau \ud \sigma
\end{equation}
where $\varphi$ is the map defined in \eqref{phi change of variables}. By a change of variables, the integral on the right-hand side of \eqref{wkpf2} can be written as
\begin{eqnarray*}
\int_{\varphi(\Omega)} \prod_{k=1}^M \tau_k^{2K/d(d+1)} \,\ud \rho \ud \tau \ud \sigma &=&  \int_{\Omega} \prod_{k=1}^M t_{l_k}^{2K/d(d+1)}\left| \det \frac{\partial \varphi}{\partial (r, t)}(r,t) \right| \,\ud r \ud t \\
&=& \int_{\Omega} \prod_{k=1}^M t_{l_k}^{2K/d(d+1)}\prod_{j=1}^{n} t_{l_{\nu_j}}^{2K/d(d+1)}\,\ud r \ud t.
\end{eqnarray*}
Arguing as in the last step of the proof of Lemma \ref{towerlem}, one may deduce $t_{l_{\nu_j}} \sim t_{l_{\nu_j}+1}$ for $1 \leq j \leq n$ provided that the parameter $\delta$ from Lemma \ref{towerlem} is chosen to be sufficiently small (depending only on the degree of $P$ and $d$). The previous expression is therefore bounded below by a constant multiple of
\begin{equation*}
\int_{\Omega} \prod_{k=1}^N t_{k}^{2K/d(d+1)}\,\ud r \ud t.
\end{equation*}
Applying Fubini's theorem and the estimates for the $\mu_P$-measure of the fibres of the $\Omega_j$, one may easily deduce the above integral is at least a constant multiple of $\alpha^N(\beta/\alpha)^{\floor{N/2}}$. Whence, combining these observations and multiplying both sides of \eqref{wkpf2} by $\alpha^{-n}(\beta/\alpha)^{-n/2}$, one arrives at the estimate 
\begin{equation*}
|X| \gtrsim \alpha^{d(d+1)/2 - M + N-n}(\beta / \alpha )^{(m-\eta)/2 + \floor{N/2}}.
\end{equation*}
Recalling $M=N-n$ and \eqref{counting variables}, this is easily seen to be the desired estimate. 
\end{proof}

To complete the proof of Theorem \ref{weakthm} it remains to prove Lemma \ref{jaclem1}.

\section{The proof of the Jacobian estimates: $\mathrm{Case}(1,d)$}

In the previous section the proof of the main theorem was reduced to establishing the pointwise estimates for the Jacobian function described in Lemma \ref{jaclem1}. Here the proof of Lemma \ref{jaclem1} in $\mathrm{Case}(1,d)$ is discussed in detail. The same arguments can be adapted to treat the remaining cases, as demonstrated in the following section. 

\begin{proof}[Proof (of Lemma \ref{jaclem1}, assuming $\mathrm{Case}(1,d)$ holds)] The arguments here, which are based primarily on those of \cite{Christ1998, Stovall2010}, are somewhat lengthy; it is convenient, therefore, to present the proof as a series of steps.

\subsection*{Compute the Jacobian matrix.} Recalling the definition of the mapping $\Phi := \Phi_N$, one may use the established index notation to express $\Phi \circ \varphi^{-1}$ as 
\begin{eqnarray*}
\Phi\circ \varphi^{-1}(\rho, \tau, \sigma) &=& x_0 -r_0P(\tau_1) + \sum_{i=1}^m \rho_{\mu_i}\big( P(\tau_{\mu_i}) - P(\tau_{\mu_i +1}) \big) \\
&& + \sum_{j=1}^n \rho_{\nu_j}\big( P(\tau_{\nu_j}+ s_{\nu_j}(\tau,\sigma)) - P(\tau_{\nu_j +1}) \big).
\end{eqnarray*}
One immediately deduces that
\begin{equation*}
\frac{\partial G_{\sigma}}{\partial \rho_{\mu_i}} (\rho, \tau) = P(\tau_{\mu_i}) - P(\tau_{\mu_i +1}) \qquad \textrm{for $i = 1, \dots, m$}
\end{equation*}
whilst 
\begin{equation*}
\frac{\partial G_{\sigma}}{\partial \rho_{\nu_j}} (\rho, \tau) = P(\tau_{\nu_j} + s_{\nu_j}(\tau,\sigma)) - P(\tau_{\nu_j +1}) \qquad \textrm{for $j = 1, \dots, n$}
\end{equation*}
which identifies $m+n$ of the columns of the Jacobian matrix. The remaining columns correspond to differentiation with respect to the $\tau$ variables and are readily computed by expressing $\Phi\circ \varphi^{-1}$ as
\begin{eqnarray*}
\Phi\circ \varphi^{-1}(\rho, \tau, \sigma) &=& x_0 + \sum_{i=1}^m \big( \rho_{\mu_i}P(\tau_{\mu_i}) - \rho^*_{\mu_i-1}P(\tau_{\mu_i - 1})\big) \\
&&+ \sum_{j=1}^n \big( \rho_{\nu_j}P(\tau_{\nu_j} + s_{\nu_j}(\tau,\sigma)) - \rho^*_{\nu_j-1}P(\tau_{\nu_j})\big) - \rho^*_MP(\tau_M)
\end{eqnarray*}
where the $\rho^*_{\mu_i-1}$, $\rho^*_{\nu_j-1}$ and $\rho^*_M$ are defined in the obvious manner; for instance, if $\rho_{\nu_j}$ corresponds to the parameter $r_k$ via the change of variables $\varphi$, then $\rho^*_{\nu_j-1}$ is understood to correspond to the parameter $r_{k-1}$. Thus for $i=1, \dots, m$ one has
\begin{equation*}
\frac{\partial G_{\sigma}}{\partial \tau_{\mu_i}} (\rho, \tau) =  \rho_{\mu_i}P'(\tau_{\mu_i}) \quad \textrm{and} \quad \frac{\partial G_{\sigma}}{\partial \tau_{\mu_i - 1}} (\rho, \tau) =  - \rho^*_{\mu_i-1}P'(\tau_{\mu_i - 1})
\end{equation*}
whilst
\begin{equation*}
\frac{\partial G_{\sigma}}{\partial \tau_{M}} (\rho, \tau) = - \rho^*_MP'(\tau_M),
\end{equation*}
accounting for a further $2m+1$ columns to the Jacobian matrix. To compute the remaining $n$ columns differentiate $G_{\sigma}$ with respect to the $\tau_{\nu_j}$ to give
\begin{eqnarray}\label{taunucolumn}
\frac{\partial G_{\sigma}}{\partial \tau_{\nu_j}} (\rho, \tau) &=& \rho_{\nu_j}P'(\tau_{\nu_j}+ s_{\nu_j}(\tau, \sigma)) - \rho^*_{\nu_j-1}P'(\tau_{\nu_j}) \\
\nonumber
&& - \frac{2K}{d(d+1)} \frac{s_{\nu_j}(\tau, \sigma)}{\tau_{\nu_j}}\rho_{\nu_j}P'(\tau_{\nu_j}+ s_{\nu_j}(\tau, \sigma))
\end{eqnarray}
for $j=1, \dots, n$. 

\subsection*{Compare $J_{\sigma}$ with $J_P$.} The estimation of the Jacobian $J_{\sigma}$ will be achieved by comparing it to the more tractable expression $J_P$, introduced in \eqref{JP}. Once such a comparison is established, $J_{\sigma}$ can then be bounded by means of the geometric inequality of Dendrinos and Wright (that is, Theorem \ref{Dendrinos Wright theorem}). This inequality is guaranteed to hold in the appropriate setting due to the reductions made earlier in the article. 

To begin, express the Jacobian determinant in the form of an integral
\begin{equation*}
J_{\sigma}(\rho, \tau)= \pm\int_{R(\tau) \times B_{\sigma}(\tau)} \wp_{\sigma}(\rho, \tau, x) \,\ud x
\end{equation*}
where $\wp_{\sigma}(\rho, \tau, x)$ is a multi-variate polynomial and 
\begin{equation*}
R(\tau) := \prod_{i=1}^{m} (\tau_{\mu_i}, \tau_{\mu_i+1}) \quad \textrm{and} \quad
B_{\sigma}(\tau) := \prod_{i=1}^{n} (\tau_{\nu_j} + s_{\nu_j}, \tau_{\nu_j+1}) 
\end{equation*}
are rectangles.\footnote{For notational convenience the dependence of $s_{\nu}$ on $(\tau, \sigma)$ has been suppressed.} The polynomial $\wp_{\sigma}(\rho, \tau, x)$ is the product of $C(\rho) = \rho_M^* \prod_{i=1}^m \rho_{\mu_i - 1}^* \rho_{\mu_i}$ and the determinant of the matrix $A_{\sigma}(\rho, \tau, x)$ obtained from original Jacobian matrix by making the following changes:
\begin{itemize}
\item The column $P(\tau_{\mu_i}) - P(\tau_{\mu_i +1})$ is replaced with $P'(x_i)$ for $i=1, \dots, m$.
\item The column $P(\tau_{\nu_j}+ s_{\nu_j}) - P(\tau_{\nu_j +1})$ is replaced with $P'(x_{m+j})$ for $j = 1, \dots, n$.
\item The columns $\rho_{\mu_i}P'(\tau_{\mu_i})$ and $- \rho^*_{\mu_i-1}P'(\tau_{\mu_i - 1})$ are replaced with $P'(\tau_{\mu_i})$ and $P'(\tau_{\mu_i - 1})$, respectively, for all $i = 1, \dots, m$. In addition, $- \rho^*_MP'(\tau_M)$ is replaced with and $P'(\tau_M)$.
\item The remaining columns $n$ are unaltered; in other words, they agree with the corresponding columns of the Jacobian matrix.
\end{itemize}

Notice the unaltered columns are those corresponding to differentiation by $\tau_{\nu_j}$ and are of the form given in \eqref{taunucolumn}. Each may be expressed as the sum of three terms
\begin{equation}\label{taylor1}
\frac{\partial G_{\sigma}}{\partial \tau_{\nu_j}} (\rho, \tau) = \sum_{i=1}^3 T^i_{\sigma, j}(\rho,\tau)
\end{equation}
where, writing $c := - 2K/d(d+1)$, 
\begin{eqnarray*}
T^1_{\sigma, j}(\rho,\tau) &:=& (\rho_{\nu_j} - \rho^*_{\nu_j-1})P'(\tau_{\nu_j}), \\
T^2_{\sigma, j}(\rho,\tau) &:=& c \frac{s_{\nu_j}(\tau, \sigma)}{\tau_{\nu_j}}\rho_{\nu_j}P'(\tau_{\nu_j}+ s_{\nu_j}(\tau, \sigma)), \\
T^3_{\sigma, j}(\rho,\tau) &:=& \rho_{\nu_j}(P'(\tau_{\nu_j}+ s_{\nu_j}(\tau, \sigma)) - P'(\tau_{\nu_j})). 
\end{eqnarray*}
The multi-linearity of the determinant and \eqref{taylor1} are now applied to express $\det A_{\sigma}(\rho, \tau, x)$ as a sum of determinants of more elementary matrices. In order to present concisely the resulting expression it is useful to introduce some notation. In particular, for $S \subseteq \mathcal{N} := \{\nu_1, \dots, \nu_n\}$, let $\Delta_S$ denote the function of $\rho$ given by 
\begin{equation*}
\Delta_S(\rho) := \prod_{\nu \in S} (\rho_{\nu} - \rho^*_{\nu - 1})
\end{equation*}
and $R_{\sigma, S}(\tau) \subset \R^{\#S}$ the rectangle
\begin{equation*}
R_{\sigma, S}(\tau) := \prod_{\nu \in S} (\tau_{\nu}, \tau_{\nu}+ s_{\nu}).
\end{equation*}

With this notation $\det A_{\sigma}(\rho, \tau, x)$ equals 
\begin{equation}\label{integrand1}
\sum_{\mathcal{S}} \Delta_{S_1}(\rho) \bigg(\prod_{\nu \in S_2}\frac{c \rho_{\nu}s_{\nu}}{\tau_{\nu}}\bigg)\bigg( \prod_{\nu \in S_3} \rho_{\nu}\bigg) \int_{R_{\sigma, S_3}(\tau)} \bigg(\prod_{\nu \in S_3}\frac{\partial}{\partial y_{\nu}}\bigg) J_P(\xi_{\mathcal{S}}(y), x)\,\ud y,
\end{equation}
at least up to a sign, where the sum ranges over all partitions $\mathcal{S}:=(S_1, S_2, S_3)$ of $\mathcal{N}$ and for any such partition $\xi_{\mathcal{S}}(y) = (\xi_{\mathcal{S},l}(\tau,\sigma, y))_{l=1}^M$ is defined by
\begin{equation*}
\xi_{\mathcal{S},l}(\tau, \sigma,y) := \left\{\begin{array}{ll}
\tau_{l} + s_{l} & \textrm{if $l \in S_2$,}\\
y_{l} & \textrm{if $l \in S_3$,} \\
\tau_{l} & \textrm{otherwise.}
\end{array} \right.
\end{equation*}
If $S_3 = \emptyset$, then the integral appearing in \eqref{integrand1} is interpreted as $J_P(\xi_{\mathcal{S}}, x)$.

The term of the sum in \eqref{integrand1} corresponding to the unique partition for which $S_1 = \mathcal{N}$ is simply $\Delta(\rho)J_P(\tau, x)$ where $\Delta(\rho) := \Delta_{\mathcal{N}}(\rho)$; the sum of the remaining terms is denoted by $E_{\sigma}(\rho, \tau, x)$. Thus, \eqref{integrand1} equals
\begin{equation}\label{error}
\Delta(\rho)J_P(\tau, x) + E_{\sigma}(\rho, \tau, x).
\end{equation}
In conclusion, the Jacobian $J_{\sigma}$ can be expressed in terms of (an integral of) the function $J_P$ together with some error term.

\subsection*{Control the error.} It will be shown that provided that $\delta$ is chosen sufficiently small, depending only on $d$ and $\deg P$, the right-hand summand of \eqref{error} is subordinate to the left-hand summand. Only a bounded number of terms of \eqref{integrand1} are non-zero and the error is therefore a sum of $O(1)$ terms which will be estimated individually. 

By the properties of the parameter tower, $s_{\nu} \lesssim \delta (\beta/\alpha)\tau_{\nu}$ for all $\nu \in \mathcal{N}$. Hence, for any $S_2 \subseteq \mathcal{N}$ one has
\begin{equation}\label{error estimate 1}
\prod_{\nu \in S_2}\frac{c \rho_{\nu}s_{\nu}}{\tau_{\nu}} \lesssim \delta^{\#S_2}(\beta/\alpha)^{\#S_2} \lesssim \delta^{\#S_2} \Delta_{S_2}(\rho)
\end{equation}
where the final inequality is due to the definition of the blue indices. A suitable error bound would follow from a similar estimate for each of the integrals appearing in \eqref{integrand1}. In particular, fixing some partition $\mathcal{S} = (S_1, S_2, S_3)$ of $\mathcal{N}$, it suffices to prove
\begin{equation}\label{error estimate 2}
\int_{R_{\sigma, S_3}(\tau)} \left|\left(\prod_{\nu \in S_3}\frac{\partial}{\partial y_{\nu}}\right) J_P(\xi_{\mathcal{S}}(y), x)\right|\,\ud y \lesssim \delta^{\#S_3} \Delta_{S_3}(\rho)|J_P(\tau, x)|.
\end{equation} 
Indeed, once \eqref{error estimate 2} is established, the error bound 
\begin{eqnarray}\label{errorest1}
|E_{\sigma}(\rho, \tau, x)| &\lesssim&  \bigg(\sum_{\substack{\mathcal{S} = (S_1, S_2, S_3) \\ S_1 \neq \mathcal{N}}} \delta^{\#S_2 + \#S_3} \prod_{j=1}^3\Delta_{S_j}(\rho)\bigg) |J_P(\tau, x)| \\
\nonumber
&\lesssim& \delta \Delta(\rho)|J_P(\tau, x)|
\end{eqnarray}
immediately follows, noting the factor $\prod_{\nu \in S_3} \rho_{\nu}$ from \eqref{integrand1} is $O(1)$ whenever it appears in a non-zero term of the sum.

If $S_3 = \emptyset$, then \eqref{error estimate 2} is trivial. Fix a partition $\mathcal{S}$ as above with $S_3$ non-empty and some $y \in R_{\sigma, S_3}(\tau)$ and consider the ratio
\begin{equation}\label{ratio1}
\left| \frac{\left(\prod_{\nu \in S_3}\frac{\partial}{\partial y_{\nu}}\right)J_P(\xi_{\mathcal{S}}(y), x)}{J_P(\xi_{\mathcal{S}}(y), x)} \right|.
\end{equation}
For notational convenience write $\xi = (\xi_1, \dots, \xi_n) := \xi_{\mathcal{S}}(y)$. Using the derivative estimate from Proposition \ref{Stovall's observation} one may bound \eqref{ratio1} by a linear combination of $O(1)$ terms (with $O(1)$ coefficients) of the form
\begin{equation}\label{complicated thing}
\bigg(\prod_{{\nu} \in T_1} y_{\nu}^{-1}\bigg)\bigg( \prod_{{\nu} \in T_2} y_{\nu}^{-\epsilon(\nu)}|y_{\nu} - \xi_{u({\nu})}|^{\epsilon(\nu) - 1}\bigg)\bigg( \prod_{{\nu} \in T_3}y_{\nu}^{-\epsilon(\nu)}|y_{\nu} - x_{v({\nu})}|^{\epsilon(\nu) - 1}\bigg)
\end{equation}
where:
\begin{itemize}
\item $(T_1, T_2, T_3)$ is a partition of $S_3$;
\item $u \colon T_2 \to \{1, \dots, M\}$ is a function with the property $u(j) \neq j$ for all $j \in T_2$; 
\item $v \colon T_3 \to \{1, \dots, n+m\}$ (with no additional conditions) and
\item $\epsilon \colon T_2 \cup T_3 \to \{0,1\}$. 
\end{itemize}  

To prove \eqref{error estimate 2} it therefore suffices to establish a suitable bound for the integral of the product of \eqref{complicated thing} and $|J_P(\xi, x)|$ over the set $R_{\sigma, S_3}(\tau)$. The first step is to estimate \eqref{complicated thing} by applying the following observations.

\begin{enumerate}[i)]
\item Given $y \in R_{\sigma, S_3}(\tau)$, by the definition of the parameter tower the estimates
\begin{eqnarray*}
y_{\nu} &\geq& \tau_{\nu} = \tau_{\nu}^{1/\kappa} \tau_{\nu}^{-2K/d(d+1)} \gtrsim \alpha\tau_{\nu}^{-2K/d(d+1)};\\
|y_{\nu} - \xi_{u({\nu})}| &\geq& |\tau_{\nu} - \tau_{u({\nu})}| - \delta \alpha (\tau_{\nu}^{-2K/d(d+1)} + \tau_{u({\nu})}^{-2K/d(d+1)}); \\
|y_{\nu} - x_{v({\nu})}| &\geq& |\tau_{\nu} - x_{v({\nu})}| - \delta \alpha \tau_{\nu}^{-2K/d(d+1)}
\end{eqnarray*}
hold for ${\nu} \in \mathcal{N}$.
\item Since the indices $l_{\nu}$ for $\nu \in \mathcal{N}$ are those that directly precede a blue index (and so $l_{\nu}$ is odd), Corollary \ref{separation} ensures $\tau_{\nu} - \tau_u \gtrsim \alpha \tau_u^{-2K/d(d+1)}$ for all $1 \leq u < \nu$. Moreover, the ordering of the variables then guarantees
\begin{equation*}
\tau_{\nu} - \tau_u \gtrsim \alpha \tau_{\nu}^{-2K/d(d+1)} \qquad \textrm{whenever $1 \leq u < \nu$}. 
\end{equation*}
\item On the other hand, since the labelling $l_k$ omits the blue indices, for any $\nu \in \mathcal{N}$ and $\nu < u \leq M$ one must have $l - l_{\nu} \geq 2$ where $l$ is the index such that $\tau_u = t_l$. Consequently, by applying Corollary \ref{separation} in this case one concludes that
\begin{equation*}
\tau_u-\tau_{\nu} \gtrsim \alpha \tau_{\nu}^{-2K/d(d+1)} \qquad \textrm{whenever $\nu < u \leq M$}.
\end{equation*}
\end{enumerate}
Combining these observations one immediately deduces that
\begin{equation*}
|y_{\nu} - \xi_{u(\nu)}| \gtrsim\alpha \tau_{\nu}^{-2K/d(d+1)}
\end{equation*}
for all $\nu \in \mathcal{N}$, provided that $\delta$ is chosen initially to be sufficiently small in the earlier application of Lemma \ref{towerlem}.

It would be useful to have a similar bound for the terms $|y_{\nu} - x_l|$. At present such an estimate is not possible due to the potential lack of separation between the $\tau_{\nu}$ and $x_l$ variables. To remedy this, temporarily assume the addition separation hypothesis 
\begin{equation}\label{sephypo1}
|\tau_{\nu} - x_l| \gtrsim \alpha\tau_{\nu}^{-2K/d(d+1)} 
\end{equation}
for all $\nu \in \mathcal{N}$ and all $1 \leq l \leq m+n$. Presently it is shown that this separation hypothesis leads to desirable control over the error term $E_{\sigma}(\rho, \tau, x)$; the following step is then to modify the existing set-up so that \eqref{sephypo1} indeed holds without the need of additional assumptions. 

The preceding discussion, together with the identity $|R_{\sigma, S}(\tau)| = \prod_{\nu \in S} s_{\nu}$, implies \eqref{complicated thing} is controlled by 
\begin{eqnarray*}
\alpha^{-\# S_3} \prod_{\nu \in S_3} \tau_{\nu}^{2K/d(d+1)} &\lesssim& \delta^{\#S_3} (\beta / \alpha)^{\#S_3/2}  |R_{\sigma, S_3}(\tau)|^{-1}\\
&\lesssim& \delta^{\#S_3} |\Delta_{S_3}(\rho)||R_{\sigma, S_3}(\tau)|^{-1}
\end{eqnarray*}
 provided that $\delta$ is chosen to be sufficiently small. Observe, both of the above inequalities are simple consequences of the definition of the blue indices. Consequently, the left-hand side of \eqref{error estimate 2} may be bounded by
\begin{equation*}
 \delta^{\#S_3} \Delta_{S_3}(\rho)\frac{1}{|R_{\sigma, S_3}(\tau)|}\int_{R_{\sigma, S_3}(\tau)}|J_P(\xi_{\mathcal{S}}, x)|\,\ud y
\end{equation*}
and so \eqref{error estimate 2}, and thence \eqref{errorest1}, would follow if 
\begin{equation*}
|J_P(\xi(y), x)| \sim |J_P(\tau, x)| \qquad \textrm{ for all $y \in R_{\sigma, S_3}(\tau)$.}
\end{equation*}
This approximation is readily deduced by combining Proposition \ref{Stovall's observation} with Gr\"onwall's inequality (for a proof of Gr\"onwall's inequality see, for instance, \cite[Chapter 1]{Tao2006}). 

Hence, the estimate \eqref{errorest1} is established under the assumption of the separation hypothesis \eqref{sephypo1}.

\subsection*{Enforce separation.} In the previous section it was shown if \eqref{sephypo1} were to hold for each $\nu \in \mathcal{N}$ uniformly over all $x = (x_1, \dots, x_{m+n}) \in R_{\sigma, S_3}(\tau)$, then by choosing $0< \delta \ll 1$ sufficiently small one may control the integrand by the easily-understood function $|\Delta(\rho)|  |J_P(\tau, x)|$. Clearly for fixed $i$ the estimate \eqref{sephypo1} cannot hold for at least one value of $l$, since as $x$ varies over $R(\tau)\times B_{\sigma}(\tau)$ some $x_l$ can stray close to $\tau_{\nu_i}$ in the boundary regions. To remedy this problem one simply removes a suitable small portion of $R(\tau) \times B_{\sigma}(\tau)$ from the boundary, observing that this can be done without greatly diminishing the size of the integral to be estimated. Given $0< \epsilon < 1/2$, $ 1 \leq i \leq m$ and $1 \leq j \leq n$, define the $\epsilon$-truncate of $R_i(\tau) := (\tau_{\mu_i}, \tau_{\mu_i+1})$ and $B_{\sigma, j}(\tau) := (\tau_{\nu_j} + s_{\nu_j}, \tau_{\nu_j+1})$ by
\begin{equation*}
R^{\epsilon}_i(\tau) := (\tau_{\mu_i} +\epsilon |R_i(\tau)|, \tau_{\mu_i+1}- \epsilon |R_i(\tau)|)
\end{equation*}
and
\begin{equation*}
B^{\epsilon}_{\sigma, j}(\tau) := (\tau_{\nu_j} + s_{\nu_j} +\epsilon |B_{\sigma, j}(\tau)|, \tau_{\nu_j+1}- \epsilon |B_{\sigma, j}(\tau)|),
\end{equation*} 
respectively. Moreover, define the $\epsilon$-truncates of the associated rectangles to be $R^{\epsilon}(\tau) := \prod_{i=1}^m R^{\epsilon}_i(\tau)$ and $B^{\epsilon}_{\sigma}(\tau) := \prod_{j=1}^n B^{\epsilon}_{\sigma, j}(\tau)$. Lemma \ref{polylem} below establishes the existence of some constant $0< c_0 < 1/2$, depending only on $d$ and $\deg P$, such that
\begin{equation}\label{jacoint1}
|J_{\sigma}(\rho, \tau)| \geq \bigg|\int_{D(\tau)} \wp_{\sigma}(\rho, \tau, x) \,\ud x\bigg| - \frac{1}{2} \int_{D(\tau)} |\wp_{\sigma}(\rho, \tau, x)| \,\ud x.
\end{equation}
where $D(\tau) := R^{c_0}(\tau) \times B^{c_0}_{\sigma}(\tau)$. 

It is easy to show that for all $x \in D(\tau)$ the condition \eqref{sephypo1} holds with a uniform constant. Observe
\begin{equation*}
|B_{\sigma, j}(\tau)| = \tau_{\nu_{j+1}} - (\tau_{\nu_j} + s_{\nu_j}) = t_{l_{(\nu_j +1)}} - t_{(l_{\nu_j} +1)},
\end{equation*}
where the brackets in the subscript are included to aid the clarity of exposition. Since $l_{\nu_j} +1$ is, by definition, a blue index it follows that $l_{(\nu_j +1)}$ is odd and, consequently, 
\begin{equation*}
|B_{\sigma, j}(\tau)| \gtrsim \alpha t_{(l_{\nu_j} +1)}^{-2K/d(d+1)} = \alpha (\tau_{\nu_j} + s_{\nu_j})^{-2K/d(d+1)}
\end{equation*}
by Corollary \ref{separation} part i). Futhermore, recalling $s_{\nu_j} \lesssim \delta (\beta/\alpha)\tau_{\nu_j} \lesssim \tau_{\nu_j}$, it follows that
\begin{equation}\label{B length}
|B_{\sigma, j}(\tau)| \gtrsim \alpha \tau_{\nu_j}^{-2K/d(d+1)}.
\end{equation}
Now suppose $x_l \in B^{c_0}_{\sigma, j_0}(\tau)$ for some fixed $j_0 \in \{1, \dots, n\}$. It is clear from the definition of the parameter domain that if $j \neq j_0$, then \eqref{sephypo1} holds for $\nu = \nu_j$. Similarly, if $x_l \in R^{c_0}_{i_0}(\tau)$ for some fixed $i_0 \in \{1, \dots, m\}$, then \eqref{sephypo1} holds for all $\nu \in \mathcal{N}$. It remains to verify \eqref{sephypo1} when $x_l \in B^{c_0}_{\sigma, j_0}(\tau)$ and $j = j_0$, but this is immediate from the definition of the truncation and the bound \eqref{B length}.

Consequently, for $x \in D(\tau)$ and $\delta$ sufficiently small \eqref{errorest1} holds and thus the estimate
\begin{equation}\label{polynomial estimate}
|\wp_{\sigma}(\rho, \tau, x)| \gtrsim |\Delta(\rho)||J_P(\tau, x)|
\end{equation}
is valid on $D(\tau)$. Furthermore, it is claimed that as $x$ varies over $D(\tau)$ the sign of $\wp_{\sigma}(\rho, \tau, x)$ is unchanged. Once this observation is established the right-hand side of \eqref{jacoint1} can be written as
\begin{equation*}
\frac{1}{2} \int_{D(\tau)} |\wp_{\sigma}(\rho, \tau, x)| \,\ud x \gtrsim |\Delta(\rho) |\int_{D(\tau)}| J_P(\tau, x)|\,\ud x
\end{equation*}
To prove the claim, note that the ordering of the components of the $(r,t) \in \Omega$ implies the sign of $V(\tau, x)$ is fixed as $x$ varies over $D(\tau)$; the geometric inequality guaranteed by Theorem \ref{Dendrinos Wright theorem} therefore ensures that the sign of $J_P(\tau, x)$ is also fixed (and is non-zero). The estimate \eqref{polynomial estimate} now implies the claim.

\subsection*{Bound $J_P$ and apply the properties of $\Omega$.} Combining the estimate guaranteed by Theorem \ref{Dendrinos Wright theorem} and the preceding observations one deduces that
\begin{equation*}
|J_{\sigma}(\rho, \tau)| \gtrsim |\Delta(\rho) | \prod_{l=1}^{M}|L_P(\tau_l)|^{1/d} \int_{D(\tau)}\prod_{k=1}^{m+n}|L_P(x_k)|^{1/d} | V(\tau, x)| \,\ud x.
\end{equation*}
Over the domain of integration the estimate 
\begin{equation*}
| V(\tau, x)| \gtrsim \alpha^{d(d-1)/2-M(M-1)/2} |V(\tau)| \prod_{l=1}^M \tau_l^{-K(d-M)/d(d+1)} \prod_{k=1}^{m+n} x_k^{-K(d-1)/d(d+1)}
\end{equation*}
is valid owing to both \eqref{sephypo1} and the additional separation enforced by truncating the set $R(\tau)$. Furthermore, the construction of the (type 1) parameter tower ensures
\begin{equation}\label{Vandermonde estimate}
| V(\tau)| \gtrsim \alpha^{M(M-1)/2} (\beta/\alpha)^{m/2} \prod_{l=1}^M \tau_l^{-K(M-1)/d(d+1)}.
\end{equation}
Since the properties of the blue intervals imply $|\Delta(\rho) |\gtrsim (\beta/\alpha)^{n/2}$, one may combine the preceding inequalities to deduce
\begin{equation}\label{Jacobian bound}
|J_{\sigma}(\rho, \tau)| \gtrsim \alpha^{d(d-1)/2}(\beta/\alpha)^{(m+n)/2} \bigg(\int_{D(\tau)}\prod_{k=1}^{d-M}x_k^{2K/d(d+1)} \,\ud x \bigg)\prod_{l=1}^{M}\tau_l^{2K/d(d+1)} .
\end{equation}
Here the approximation $L_P(t) \sim t^{K}$ has been applied, which was a consequence of the decomposition theorem. 

Finally, the integral on the right-hand side of the above expression is easily seen to satisfy
\begin{equation*}
\int_{D(\tau)}\prod_{k=1}^{d-M}x_k^{2K/d(d+1)} \,\ud x \gtrsim \alpha^{d - M},
\end{equation*}
concluding the proof.


\end{proof}

It remains to state and prove the lemma which justifies the estimate \eqref{jacoint1}. In general, for $0 < \epsilon < 1/2$ the $\epsilon$-truncation $I^{\epsilon}$ of a finite open interval $I = (a, b)$ is defined as $I^{\epsilon} := (a +\epsilon(b-a), b - \epsilon(b-a))$. If $I_1, \dots, I_K$ is a family of finite open intervals, the $\epsilon$-truncation $R^{\epsilon}$ of the associated rectangle $R := \prod_{j=1}^K I_j$ is defined simply by $R^{\epsilon} := \prod_{j=1}^K I_j^{\epsilon}$.

\begin{lem}\label{polylem} Given any $M, K \in \N$ there exists a constant $0 < c_{M,K} < 1/2$ with the following property. For all $0< \epsilon < c_{M,K}$ there exists $C_{M,K}(\epsilon) >0$ such that for any collection $I_1, \dots, I_K$ of finite open intervals with associated rectangle $R$ one has
\begin{equation*}
\int_{R \setminus R^{\epsilon}} |p(x)| \,\ud x \leq C_{M,K}(\epsilon) \int_{R^{\epsilon}} |p(x)| \,\ud x
\end{equation*}
whenever $p$ is a polynomial of degree at most $M$ in $x =(x_1, \dots, x_K)$. Moreover, $\lim_{\epsilon \rightarrow 0} C_{M,K}(\epsilon) = 0$ for any fixed $M, K$.
\end{lem}

Once the lemma is established, taking $n,m $ and $M$ to be as defined in the previous proof and $K := m + n$, the inequality \eqref{jacoint1} (at least in $\mathrm{Case}(1,d)$) follows by choosing $c_0$ sufficiently small so that $0 < C_{M,K}(c_0) < 1/2$. 

\begin{proof}[Proof (of Lemma \ref{polylem})] By homogeneity it suffices to consider the case $I_1 = \dots = I_K = (0,1)$ and a simple inductive procedure further reduces the problem to the case $K = 1$. Fixing $M$ and letting $I = (0,1)$, the proof is now a simple consequence of the equivalence of norms on finite-dimensional spaces: if $C_M < \infty$ is defined to be the supremum of the ratio $\|p\|_{L^{\infty}(I)} / \|p\|_{L^1(I)}$ over all polynomials of degree at most $M$, then 
\begin{equation*}
\int_{I \setminus I^{\epsilon}} |p(x)| \,\ud x \leq 2\epsilon C_M \bigg(\int_{I \setminus I^{\epsilon}} |p(x)| \,\ud x + \int_{I^{\epsilon}} |p(x)| \,\ud x \bigg).
\end{equation*}
Provided that $0 < \epsilon < C_M/2$ one may take $C_{M,1}(\epsilon) := 2\epsilon C_M / (1 - 2\epsilon C_M)$, completing the proof. 
\end{proof}

\section{The proof of the Jacobian estimates: $\mathrm{Case}(2,d+1)$ and $\mathrm{Case}(2, d)$}

The argument used to prove Lemma \ref{jaclem1} in $\mathrm{Case}(1,d)$ can easily be adapted to establish the result in the remaining cases. The necessary modifications are sketched below; the precise details are left to the patient reader.

\subsection*{Adapting the arguments to $\mathrm{Case}(2,d+1)$.} To prove the inequality in $\mathrm{Case}(2,d+1)$ only a minor modification of the preceding argument is needed. Notice by the minimality of the parameter $N$ defined in \eqref{defn N} it follows that the index $N$ is red and so $\mu_m = M$. Here $\Phi \circ \varphi^{-1}$ maps into $\R^d \times [1,2]$ and is given by
\begin{equation*}
\Phi \circ \varphi^{-1}(\rho, \tau, \sigma) = \left( \begin{array}{c}
\Psi_N(x_0, r_0; \varphi^{-1}(\rho, \tau, \sigma))\\
\rho_{\mu_m} \end{array}\right)
\end{equation*} 
where
\begin{eqnarray*}
\Psi_N(x_0, r_0; \varphi^{-1}(\rho, \tau, \sigma)) &=& x_0 -r_0P(\tau_1) + \sum_{i=1}^{m-1} \rho_{\mu_i}\big( P(\tau_{\mu_i}) - P(\tau_{\mu_i +1}) \big) \\
 &&+ \sum_{j=1}^n \rho_{\nu_j}\big( P(\tau_{\nu_j}+ s_{\nu_j}) - P(\tau_{\nu_j +1}) \big) + \rho_{\mu_m}P(\tau_{\mu_m}).
\end{eqnarray*}

 The Jacobian matrix is now a $(d+1)\times (d+1)$ matrix. The columns given by differentiating $G_{\sigma}$ with respect to $\rho_{\mu_i}$ are
\begin{equation*}
\left(\begin{array}{c}
P(\tau_{\mu_i}) - P(\tau_{\mu_i+1}) \\
0
\end{array}\right) \ \ \textrm{for $j = 1, \dots, m-1$ and} \ \ \left(\begin{array}{c}
P(\tau_{\mu_m}) \\
1
\end{array}\right).
\end{equation*}
For remaining columns, the first $d$ components are precisely the components of the corresponding columns in the previous case and the $d+1$ component is 0. Expanding the determinant across row $(d+1)$, the methods used earlier in the proof can be applied to deduce 
\begin{equation*}
J_{\sigma}(\rho, \tau) = \pm\int_{R(\tau)\times B_{\sigma}(\tau)} \wp_{\sigma}(\rho, \tau, x) \,\ud x
\end{equation*}
where $\wp_{\sigma}(\rho, \tau, x) $ is the determinant of a $d \times d$ matrix and 
\begin{equation*}
R(\tau) := \prod_{i=1}^{m-1} (\tau_{\mu_i}, \tau_{\mu_i+1}); \qquad B_{\sigma}(\tau) := \prod_{j=1}^n (\tau_{\nu_{j}} + s_{\nu_j}, \tau_{\nu_{j}+1}).
\end{equation*}
The key difference is now the integral is over a rectangle of dimension $m+n -1$ (rather than $m+n$). Define the truncated domain $D(\tau)$ in analogous manner to the previous case. Notice from \eqref{counting variables} it follows that $d-M = m+n -1$, which is precisely the dimension of the set $D(\tau)$ in the present situation. Arguing as before, the inequality \eqref{Jacobian bound} also holds in this setting and from this one obtains the required estimate.  

\subsection*{Adapting the argument to $\mathrm{Case}(2, d)$.}

Here the map $\Phi \circ \varphi^{-1}$ is given by 
\begin{eqnarray*}
\Phi\circ \varphi^{-1}(\rho, \tau, \sigma) &=& y_0 + \sum_{i=1}^m \rho_{\mu_i}\big( P(\tau_{\mu_i}) - P(\tau_{\mu_i +1}) \big) \\
&& + \sum_{j=1}^n \rho_{\nu_j}\big( P(\tau_{\nu_j}+ s_{\nu_j}(\tau,\sigma)) - P(\tau_{\nu_j +1}) \big)
\end{eqnarray*}
and thus the columns of the Jacobian matrix essentially agree with those of $\mathrm{Case}(1,d)$, with the exception that now there is no column corresponding to $\partial G_{\sigma} / \partial \tau_{\mu_1 - 1}$.

The above arguments now carry through almost verbatim; the only substantial difference in this situation is that the Vandermonde estimate \eqref{Vandermonde estimate} becomes
\begin{equation*}
|V(\tau)| \gtrsim \alpha^{M(M-1)/2} (\beta/\alpha)^{(m-1)/2}
\end{equation*}
due to the fact that the parameter tower in this situation is of type 2, as opposed to type 1 in both of the previous cases.

\section{A final remark}

\begin{rmk}\label{strong type remark} In the introduction the possibility of strengthening the restricted weak-type $(p_1, q_1)$ estimate from Proposition \ref{weakthm} to a strong-type estimate was discussed. It was remarked that the strong-type estimate in dimension $d=2$ follows from a result of Gressman \cite{Gressman2013}, but can also be established by combining the analysis contained within the present article with an extrapolation method due to Christ \cite{Christb} (see also \cite{Stovall2009}). Here some further details are sketched. The key ingredients in Christ's extrapolation technique are certain `trilinear' variants of the estimates \eqref{equivweake} and \eqref{equivweakf}. Recall, to prove the weak-type bound it sufficed to show \emph{either} \eqref{equivweake} \emph{or} \eqref{equivweakf} holds since both these estimates are equivalent. This equivalence breaks down when one passes to the trilinear setting and to establish the strong-type inequality one must prove \emph{both} the trilinear version of \eqref{equivweake} \emph{and} the trilinear version of \eqref{equivweakf} hold. This can be achieved in the $d=2$ case by introducing an ``inflation'' argument (see \cite{Christ} and also \cite{Gressman2006}). One may attempt to apply the same techniques in higher dimensions but now the Jacobian arising from the inflation is rather complicated. The question of whether or not this Jacobian can be effectively estimated remains unresolved. It is possible that the inflation argument is not required when $d$ belongs to a certain congruence class modulo $3$ and potentially the strong-type bound could be established more directly from existing arguments in this situation. 

\end{rmk}

\appendix

\section*{Appendix: The method of Dendrinos and Stovall}

This final section details the construction of the sequence of sets $\{U_k\}_{k=1}^{\infty}$ featured in Lemma \ref{towerlem}. The argument here is due to Dendrinos and Stovall \cite{Dendrinos}. At this point some preliminary definitions and remarks are pertinent. Observe
\begin{equation*}
\langle A\chi_E\,,\, \chi_F \rangle = \int_{\Sigma} \chi_F(\pi_1(x,r,t)) \chi_E(\pi_2(x,r,t)) \lambda_P(t)\,\ud x\ud r \ud t
\end{equation*}
where $\Sigma := \R^d \times [1,2] \times I$ and $\pi_1 \colon \Sigma \to \R^d \times [1,2]$ and $\pi_2 \colon \Sigma \to \R^d$ are the mappings
\begin{equation*}
\pi_1(x,r,t) := (x,r), \qquad \pi_2(x,r,t) := x - rP(t).
\end{equation*}
Define the $\pi_j$-fibres to be the sets $\pi_j^{-1}\circ\pi_j(x,r,t)$ for $(x,r,t) \in \Sigma$ and $j = 1,2$. Thus, the $\pi_1$-fibres form a partition of $\Sigma$ into a continuum of curves (which are simply parallel lines) whilst the $\pi_2$-fibres partition $\Sigma$ into a continuum of 2-surfaces. 
Writing 
\begin{equation*}
U := \pi_1^{-1}(F) \cap \pi_2^{-1}(E) = \left\{ (x,r,t) \in \Sigma : \pi_1(x,r,t) \in F\textrm{ and } \pi_2(x,r,t) \in E  \right\}
\end{equation*}
it follows that
\begin{equation*}
\langle A\chi_E\,,\, \chi_F \rangle = \int_{\Sigma} \chi_U(x,r,t)\lambda_P(t)\,\ud x\ud r \ud t.
\end{equation*}
The sets $\{U_k\}_{k=0}^{\infty}$ are defined recursively. To construct the initial set $U_0$, let
\begin{equation*}
B_0 := \{(x,r,t) \in U : 0 < t < (\alpha/2\kappa)^{\kappa}\}.
\end{equation*}
Then, recalling the definition of $\lambda_P$ and applying Fubini's theorem, it follows that
\begin{eqnarray*}
\int_{\Sigma} \chi_{B_0}(x,r,t) \,\lambda_P(t)\ud x \ud r \ud t &=& \int_{F} \int_0^{(\alpha/2\kappa)^{\kappa}} \chi_E(x - rP(t))\, \lambda_P(t)\ud t \ud x \ud r \\
&\leq& \frac{1}{2}\alpha|F| =  \frac{1}{2} \int_{\Sigma} \chi_U(x,r,t)\,\lambda_P(t)\ud x \ud r \ud t.
\end{eqnarray*}
Define $U_0 := U\setminus B_0$ so that
\begin{equation*}
\int_{\Sigma} \chi_{U_0}(x,r,t) \,\lambda_P(t)\ud x \ud r \ud t \geq \frac{1}{2}\int_{\Sigma} \chi_{U}(x,r,t) \,\lambda_P(t)\ud x \ud r \ud t.
\end{equation*}
Note that this definition will ensure property iii) holds for the sequence of refinements. Now suppose the set $U_{k-1}$ has been defined for some $k \geq 1$ and satisfies the conditions stipulated in the proof of Lemma \ref{towerlem}.

\subsection*{Case $k \equiv d \mod 2$.} In order to ensure the property \eqref{ds2} holds in this case, the following refinement procedure is applied. Let $B_{k-1}$ denote the set
\begin{equation*}
\left\{ (x,r,t) \in U_{k-1} : \int_1^2\int_I \chi_{U_{k-1}}(x - rP(t) + \rho P(\tau), \rho, \tau) \,\lambda_P(\tau)\ud \tau\ud \rho \leq 4^{-(k+1/2)} \beta \right\}.
\end{equation*}
The map $(\rho, \tau) \mapsto (x - rP(t) + \rho P(\tau),\rho, \tau)$ parametrises the fibre $\pi_2^{-1}(\pi_2(x,r, t))$ and so $B_{k-1}$ is precisely the set of all points belonging to $\pi_2$-fibres which have a ``small'' intersection with $U_{k-1}$. Removing the parts of $U_{k-1}$ lying in these fibres should not significantly diminish the measure of the set and indeed, by Fubini's theorem and a simple change of variables,
\begin{eqnarray}
\nonumber
\int_{\Sigma} \chi_{B_{k-1}}(x,r,t) \,\lambda_P(t)\ud t \ud x\ud r &=& \int_{\R^n\times[1,2]}\int_I \chi_{B_{k-1}}(x+rP(t),r,t) \,\lambda_P(t)\ud t \ud x\ud r \\
\nonumber
&\leq& \int_{\{ x \in E \,:\, T_{k-1}(x) \leq  4^{-(k+1/2)}\beta\}} T_{k-1}(x)\, \ud x \\
\nonumber
&\leq&  4^{-(k+1/2)}\beta|E| \\
\label{Dendrinos and Stovall 1}
&\leq& \frac{1}{2}\int_{\Sigma} \chi_{U_{k-1}}(x,r,t) \,\lambda_P(t)\ud x \ud r \ud t
\end{eqnarray}
where 
\begin{equation*}
T_{k-1}(x) := \int_1^2\int_I \chi_{U_{k-1}}(x+\rho P(\tau),\rho,\tau) \,\lambda_P(\tau)\ud \tau\ud \rho.
\end{equation*}
Note that the inequality \eqref{Dendrinos and Stovall 1} is due to property i) of the sets $U_j$ for $1 \leq j \leq k-1$, stated in Lemma \ref{towerlem}.  
Thence, letting $U_{k-1}' := U_{k-1} \setminus B_{k-1}$ it follows that
\begin{equation}\label{U bound}
\int_{\Sigma} \chi_{U_{k-1}'}(x,r,t) \,\lambda_P(t)\ud x \ud r \ud t \geq \frac{1}{2}\int_{\Sigma} \chi_{U_{k-1}}(x,r,t) \,\lambda_P(t)\ud x \ud r \ud t.
\end{equation}
Now, recalling $I = (a,b)$, define $B_{k-1}'$ to be the set
\begin{equation*}
\left\{ (x,r,t) \in U_{k-1}' : \int_1^2\int_t^b \chi_{U_{k-1}'}(x - rP(t) + \rho P(\tau),\rho, \tau) \,\lambda_P(\tau)\ud \tau\ud \rho \leq 4^{-(k+1)}\beta \right\}.
\end{equation*}
Given $x \in \pi_2(U_{k-1}')$, the fibre-wise nature of the definition of $U_{k-1}'$ implies for $(x+rP(t),r,t) \in U_{k-1}'$ if and only if $(x+rP(t),r,t) \in U_{k-1}$ and consequently
\begin{equation}\label{ref1}
\int_1^2\int_I \chi_{U_{k-1}'}(x+rP(t),r,t) \,\lambda_P(t)\ud t\ud r  \geq 4^{-(k+1/2)} \beta.
\end{equation}
On the other hand,
\begin{equation}\label{ref2}
\int_1^2\int_I \chi_{B_{k-1}'}(x+rP(t),r,t) \, \lambda_P(t)\ud t\ud r = 4^{-(k+1)}\beta.
\end{equation}
Indeed, the left-hand side can be expressed as
\begin{equation*}
\nu_P\big(\big\{(r,t) \in K(x) : \nu_P\big(K(x) \cap ([1,2] \times (t,b))\big) \leq 4^{-(k+1)}\beta\big\}\big)
\end{equation*} 
for $K(x) \subseteq [1,2] \times I$ a measurable subset. The identity \eqref{ref2} is now a consequence of the fact that for any measure $\nu$ on $\R^2$ which is, say, absolutely continuous with respect to Lebesgue measure, 
\begin{equation*}
\nu\big(\big\{ (r,t) \in K : \nu\big(K \cap (\R \times (t, \infty))\big) \leq u \big\}\big) = u
\end{equation*}
 for all $0< u < \nu(K)$ and all $K \subseteq \R^2$ measurable.

Thence, combining \eqref{ref1} and \eqref{ref2} it follows that
\begin{equation*}
\int_1^2\int_I \chi_{B_{k-1}'}(x+rP(t),r,t) \,\lambda_P(t)\ud t\ud r \leq \frac{1}{2} \int_1^2\int_I \chi_{U_{k-1}'}(x+rP(t),r,t) \lambda_P(t)\ud t\ud r
\end{equation*}
whenever $x \in \pi_2(U_{k-1}')$. Defining $U_{k}:= U_{k-1}' \setminus B_{k-1}'$, one observes that
\begin{eqnarray*}
\int_{\Sigma} \chi_{U_{k}}(x,r,t) \,\lambda_P(t) \ud x\ud r\ud t  &=& \int_{\pi_2(U_{k-1}')} \int_1^2\int_I \chi_{U_k}(x+rP(t),r,t) \,\lambda_P(t)\ud t\ud r \ud x \\
&\geq& \frac{1}{2}\int_{\pi_2(U_{k-1}')} \int_1^2\int_I \chi_{U_{k-1}'}(x+rP(t),r,t) \,\lambda_P(t)\ud t\ud r \ud x \\
&=& \frac{1}{4}\int_{\Sigma} \chi_{U_{k-1}}(x,r,t) \,\lambda_P(t)\ud x\ud r\ud t.
\end{eqnarray*}
Moreover, the set $U_k$ is easily seen to satisfy \eqref{ds2}.

\subsection*{Case $k \not \equiv d \mod 2$.} It remains to define the set $U_k$ under the assumption $k \not \equiv d \mod 2$, ensuring property \eqref{ds1} is satisfied. Here one is concerned with the fibres of the map $\pi_1$. Define
\begin{equation*}
B_{k-1} := \left\{ (x,r,t) \in U_{k-1} : \int_I \chi_{U_{k-1}}(x,r, \tau) \,\lambda_P(\tau)\ud \tau \leq  4^{-(k+1/2)} \alpha \right\}.
\end{equation*}
Notice that the map $\tau \mapsto (x,r, \tau)$ parametrises the fibre $\pi_1^{-1}(\pi_1(x,r,t))$ and so $B_{k-1}$ is the collection of all points $(x,r,t)$ in $U_{k-1}$ which belong to $\pi_1$-fibres which have a ``small'' intersection with $U_{k-1}$. Reasoning analogously to the previous case, if one defines $U_{k-1}' := U_{k-1} \setminus B_{k-1}$ it follows that \eqref{U bound} holds in this case. Finally, let 
\begin{equation*}
B_{k-1}' := \left\{ (x,r,t) \in U_{k-1}' : \int_t^b \chi_{U_{k-1}'}(x,r, \tau) \,\lambda_P(\tau)\ud \tau \leq  4^{-(k+1)} \alpha \right\}
\end{equation*}
and $U_{k} := U_{k-1}'\setminus B_{k-1}'$. Again arguing as in the previous case, it follows that
\begin{equation*}
\int_{\Sigma} \chi_{U_{k}}(x,r,t) \,\lambda_P(t)\ud x \ud r \ud t \geq \frac{1}{4}\int_{\Sigma} \chi_{U_{k-1}}(x,r,t) \,\lambda_P(t)\ud x \ud r \ud t.
\end{equation*}
This recursive procedure defines a sequence of sets with all the desired properties.

\bibliography{Reference}
\bibliographystyle{amsplain}

\end{document}